\documentclass[12pt,a4paper]{article}
\usepackage{xcolor}

\usepackage{cancel}
\usepackage[normalem]{ulem}
\usepackage{subfigure,caption
}
\usepackage{mysec}
\usepackage{amssymb, amsmath, amsthm,graphicx, fancyhdr,bm,bbm}
\usepackage{hyperref}
\DeclareGraphicsExtensions{.pdf,.png,.jpg,.eps}

\usepackage[thinc]{esdiff}
\usepackage{amsfonts}
\usepackage{color}
\definecolor{dkgreen}{rgb}{0,0.6,0}
\definecolor{gray}{rgb}{0.5,0.5,0.5}
\definecolor{mauve}{rgb}{0.58,0,0.82}
\usepackage{listings} 
\newtheorem{theorem}{Theorem}[section]
\newtheorem{lemma}[theorem]{Lemma}
\newtheorem{proposition}[theorem]{Proposition}

\theoremstyle{remark}
\newtheorem{remark}{Remark}[section]
\newtheorem{assumption}{Assumption}[section]

\usepackage{enumerate}

\numberwithin{equation}{section}

\newcommand{\NN}{\mathbb{N}}

\newcommand{\RR}{\mathbb{R}}
\newcommand{\CC}{\mathbb{C}}
\newcommand{\rme}{\mathrm{e}}
\newcommand{\rmd}{\mathrm{d}}
\newcommand{\rmi}{\mathrm{i}}
\newcommand{\PP}{{\mathsf P}}

\newcommand{\EE}{{\mathsf E}}
\newcommand{\Var}{{\mathsf{Var}}}
\newcommand{\Cov}{{\mathsf{Cov}}}

\usepackage{tcolorbox}

\newcommand{\myp}{\mbox{$\:\!$}}
\newcommand{\mypp}{\mbox{$\;\!$}}

\newcommand{\myn}{\mbox{$\;\!\!$}}
\newcommand{\mynn}{\mbox{$\:\!\!$}}

\begin{document}
\title{Limit Shape of the Generalized Inverse Gaussian-Poisson Distribution}
\author{Leonid V.\ Bogachev\thanks{\,\href{mailto:L.V.Bogachev@leeds.ac.uk}{L.V.Bogachev@leeds.ac.uk}, ORCID \href{https://orcid.org/0000-0002-2365-2621}{0000-0002-2365-2621}}, \,Ruheyan Nuermaimaiti\thanks{\,Corresponding author, \href{mailto:mmrn@leeds.ac.uk}{mmrn@leeds.ac.uk}, ORCID \href{https://orcid.org/0000-0002-5764-9949}{0000-0002-5764-9949}}\ \,and Jochen Voss\thanks{\,\href{mailto:J.Voss@leeds.ac.uk}{J.Voss@leeds.ac.uk}, ORCID \href{https://orcid.org/0000-0002-2323-3814}{0000-0002-2323-3814}}\\[.3pc]
\it School of Mathematics, University of Leeds, Leeds LS2 9JT, UK
}

\date{\vspace{-5ex}}

\maketitle

\begin{abstract}
The generalized inverse Gaussian-Poisson (GIGP) distribution proposed by Sichel in the 1970s has proved to be a flexible fitting tool for diverse frequency data, collectively described using the item production model. In this paper, we identify the limit shape (specified as an incomplete gamma function) of the properly scaled diagrammatic representations of random samples from the GIGP distribution (known as Young diagrams). We also show that fluctuations are asymptotically normal and, moreover, the corresponding empirical random process is approximated via a rescaled Brownian motion in inverted time, with the inhomogeneous time scale determined by the limit shape. Here, the limit is taken as the number of production sources is growing to infinity, coupled with an intrinsic parameter regime ensuring that the mean number of items per source is large. More precisely, for convergence to the limit shape to be valid, this combined growth should be fast enough. In the opposite regime referred to as ``chaotic'', the empirical random process is approximated by means of an inhomogeneous Poisson process in inverted time. These results are illustrated using both computer simulations and some classic data sets in informetrics.
\end{abstract}


\medskip\noindent
\emph{MSC\,2020:} Primary 62E17; Secondary 62P25


\section{Introduction}

In many applied situations, one deals with \emph{count data} in the form of sample frequencies of occurrence in one of the countably many categories (``boxes''), say, labelled by index $j\in\NN=\{1,2,\dots\}$ or $j\in\NN_0=\{0,1,2,\dots\}$. It is often appropriate to interpret occurrence in each box $j$ as the corresponding number of batched ``items'' produced by one out of the plurality of contributing ``sources''; for example, falling into box $j=0$ is interpreted as no items produced by the source. The observed data set is then the vector $(M_j)$ of the recorded counts in each box out of the total number of sources $M=\sum_j M_j$, with the total number of items $N=\sum_j j\myp M_j$.

Diverse real-life examples of such a scenario include: abundance data of various species such as butterflies, with  different species treated as sources and the observed counts as items; the number of followers (items) of different accounts in Twitter (sources); repeat-buying data, with the number of units bought (items) by households (sources); the number of papers (items) produced by authors (sources); etc.\footnote{For more examples and further references, see, e.g., a monograph by Egghe \cite{Egghe} or a review paper by Clauset et al.\ \cite{CSN}.} In the latter example, papers may themselves play the role of sources, with citations as items. Altogether, this forms an interesting triangle of relationships, \emph{authors--papers--citations (APC)}, which is one of the main subjects of investigation in \emph{informetrics}  \cite{Egghe}.

A natural objective with such types of data is to explain the observed (relative) frequencies $(M_j/M)$ by fitting a suitable distributional model $(f_j)$, preferably possessing some conceptual foundation and applicable to a variety of use cases. Of course, in any real-life data set the numbers $M_j$ will reduce to zero for values $j$ big enough, but this is reconciled with the modelling prediction $M_j/M\approx f_j$ simply by the fact that the theoretical frequencies $f_j$ tend to zero as $j\to\infty$. However, adequate modelling of the long-tail frequencies is of importance in relation to understanding the behavior of extreme values in the count data (e.g., untypically high citations).

A celebrated example of a theoretical frequency model is the \emph{power law}, first proposed by Lotka \cite{Lotka} to describe the publication statistics in chemistry and physics, and based on the empirical observation that the sample  frequencies $M_j/M$
approximately follow a power law distribution of the form $f_j=c_\alpha\myp j^{-\alpha}$ ($j\in\NN$), with some exponent $\alpha>1$
and the normalization constant\footnote{The normalization constant $c_\alpha$ is expressed via the  Riemann zeta function,  $c_\alpha^{-1}=\sum_{j=1}^\infty j^{-\alpha}=\zeta(\alpha)$. In real-life examples, the power law exponent is typically in the range $2<\alpha<3$ \cite{CSN}.} $c_\alpha>0$ such that
$\sum_{j=1}^\infty f_j=1$ \cite{Coile,Egghe}.

An evident heuristic tool to fit a power law model to the count data is by looking at the frequency plots (e.g., histograms) with logarithmic scales on both axes, whereby one seeks a straight-line fit, with the slope $-\alpha$ \cite{Nicholls}. An alternative approach \cite{CSN}, which provides the helpful smoothing of the discrete data, is via the complementary cumulative frequencies
$\bar{F}_j=\sum_{\ell\ge j} f_\ell$, where, using again the log-log plots, a good fit corresponds to a straight line, with slope $1-\alpha$. More formally, the model can
be fitted using standard statistical methods such as the maximum likelihood or ordinary
least squares estimation \cite{Nicholls}.

The conventional explanation of universality of the power law is based on the principle of \emph{cumulative advantage}, also expressed as the catchphrase ``success breeds success'', originally coined in the context of scientific productivity \cite{Price1965,Price,ER} (see also a more recent review \cite{Huber} with a critique of  cumulative advantage).
Unfortunately, the utility of Lotka's power law for real data modelling is often limited by fitting to the data well only on a reduced range of count values, requiring a truncation of lower values (see an extensive discussion in \cite{CSN}) or of higher (long-tail) values better described by a \emph{stretched-exponential law} \cite{Sornette}.

Numerous other attempts to fit theoretical distributions to a variety of count data sets included using the negative binomial distribution, the modified geometric distribution, the beta binomial distribution, and many more \cite{Johnson1} (see the discussion and further references in \cite{Sichel1985,Huber}), however, none of these distributional families proved to be sufficiently ``universal'' in explaining diverse count data sets, often failing to capture some characteristic features such as modality and long-tail behavior.

In a series of papers, Sichel \cite{Sichel1971,Sichel1973a, Sichel1974,Sichel1975, Sichel1982,Sichel1985} introduced and developed the so-called \emph{generalized inverse Gaussian-Poisson (GIGP)} model, proposed in an attempt to grasp
a plausible production of items by respecting statistical differences in the individual productivity
of sources (e.g., papers and authors, respectively). More precisely, a source is assumed to produce items according to a Poisson law with rate $\lambda$,
which is itself random with a specific choice of the \emph{generalized inverse Gaussian (GIG)}
distribution density \cite{Sichel1971,Johnson}. In other words, GIGP distribution is a mixed Poisson distribution under the GIG mixing density (see \cite{Gupta} for a survey of Poisson mixture models for long-tailed count data). Sichel applied his GIGP model to a great variety of use cases and multiple data sets, from sentence-lengths and word frequencies in written prose \cite{Sichel1974,Sichel1975} to number of stones found in diamondiferous deposits \cite{Sichel1973a} and scientific production (papers and/or citations) \cite{Sichel1985}. These examples have  demonstrated a remarkable flexibility and versatility of the GIGP distribution family.

In a more recent development, Yong \cite{Yong} proposed to use combinatorial models of random integer partitions to mimic citation count data, where the constituent parts of the integer partition represent the author's papers with the corresponding numbers of citations, respectively. The main perceived advantage of this approach was to leverage the knowledge of so-called \emph{limit shape} for suitably scaled \emph{Young diagrams} visualizing parts in the (random) integer partition, which would then enable one to estimate statistically some citation metrics such as the \emph{$h$-index}\footnote{The $h$-index is defined as the maximum number $h$ of the  author's papers, each one cited at least $h$ times.} introduced by Hirsch \cite{Hirsch}. Specifically, noting that the $h$-index corresponds geometrically to the location  with equal coordinates at the upper boundary of the Young diagram, and using an explicit equation for the limit shape under
the scaling $\sqrt{N}$ along both axes, where $N\gg1$ is the total number of citations \cite{vershik1996statistical, Pittel}, Yong came up with a simple estimate of the $h$-index, $h\approx 0.54\,\sqrt{N}$, which he then tested using several data sets of mathematical citations~\cite{Yong}.

In this paper, we apply the notion of limit shape to random samples from the GIGP distribution, taking as a large parameter its expected value $\eta$ (together with the number of sources $M$). Under a suitable normalization, we  obtain an explicit formula for the limit shape, given by the incomplete (upper) gamma function, $\varphi_\nu(x)=\int_x^\infty\mynn
s^{\nu-1}\,\rme^{-s}\,\rmd{s}$ ($x>0$), indexed by the shape parameter $\nu\ge-1$ of the GIGP distribution. In terms of empirical data analysis, this amounts to the corresponding scaling of the complementary cumulative frequency plots, which facilitates a quick visual check of the goodness-of-fit of the GIGP model, with an additional insight informed by the recommended scaling. A more careful analysis of the error bounds, based on
asymptotic confidence intervals, is made possible by virtue of our result on asymptotically Gaussian fluctuations around the limit shape $\varphi_\nu(x)$.

As follows from the predicted limit shape $\varphi_\nu(x)$, the upper tail of the GIGP model has a power-modulated exponential decay, thus strongly deviating from the power law behavior. In most practical examples, this discrepancy is not essential because of scarcity (or lack) of higher counts; however, the question of adequate modelling of extreme values is interesting, with the stretched exponential model mentioned above being an attractive alternative \cite{Sornette}.

Note that the insightful property of having a limit shape with a non-trivial scaling is not automatic for count data models; for instance, a notable exception is the power law distribution $f_j=c\mypp j^{-\alpha}$ because of being scale free: indeed, for any scaling factor $A>0$, we have $A^{\alpha}f_{\myn Ax}\equiv f_x$ ($x>0$).

In this regard, let us mention the \emph{generalized power law (GPL)} model recently introduced and studied by Nuermaimaiti et al.\ \cite{NBV}, which aims to bridge small values of counts (frequently truncated under the power law fit) and a power type upper tail. The conceptual justification of the GPL model is also based on the mixing idea as in Sichel \cite{Sichel1974}, but under the different choices  of the source production law (geometric instead of Poisson) and the mixing density (a beta distribution instead of the GIG one). The limit shape in the GPL model exists and is given by $\varphi(x;\alpha)=(1+x)^{1-\alpha}$ ($x\ge0$), where $\alpha>0$ is the shape parameter of GPL. Although the long tail of the GPL limit shape has the same power decay as in the power law case above, the key difference is that GPL is not scale free, which ensures that the corresponding scaling is non-trivial.

The rest of the paper is organized as follows. In Section \ref{sec:2.1}, we define our main model of item production and set out basic notation and some elementary relations. The notions of Young diagrams and limit shape are introduced in Section \ref{sec:2.2}, illustrated in Section \ref{sec:2.3} via the classic model of integer partitions. In Section \ref{sec:3.1}, the GIGP distribution  is defined with parameters $\nu\in\RR$, $\alpha>0$ and $0<\theta<1$, augmented in Section \ref{sec:3.2} by specification in the boundary case $\alpha=0$, and followed in Section \ref{sec:3.3} by the asymptotic analysis of the GIGP expected value in the desired limit $\theta\to1-$ (Proposition \ref{pr:eta}). In particular, this analysis explains why we impose a restriction $\nu\ge-1$.

After these preparations, the limit shape problem is addressed in Section \ref{sec:4}. First, in Section \ref{sec:4.1} the suitable scaling coefficients $A$ and $B$ are defined (Assumption \ref{as:AB}), and our main result is stated as Theorem \ref{th:main} about uniform convergence in probability (for $x\ge \delta$, with any $\delta>0$) of rescaled Young diagrams to the limit shape $\varphi_\nu(x)=\int_x^\infty \!s^{\nu-1}\mypp \rme^{-s}\,\rmd{s}$ ($x>0$), illustrated using computer simulations for two example cases, $\nu=0.5$ and $\nu=-0.5$. Importantly, for this result to be valid, the $y$-scaling coefficient $B$ must grow unboundedly (Assumption \ref{as:B}), which in turn implies that the number of sources $M$ in the item production model should be large enough. Convergence to $\varphi_\nu(x)$ is first shown in Section \ref{sec:4.2} for the expectation of Young diagrams (Theorem \ref{th:conv_exp}). Pointwise convergence in probability is then established in Section \ref{sec:4.3} (Theorem \ref{th:conv_Y}). In Section \ref{sec:4.4}, we construct a suitable martingale in inverted time $t=1/x$ (Lemma \ref{lm:W}), which enables us to prove Theorem \ref{th:main} by applying the Doob--Kolmogorov submartingale inequality \cite[Sec.\,\ref{sec:4.5}]{Yeh}. In Section \ref{sec:4.6}, an alternative proof of Theorem \ref{th:main} is given using the interpretation of the Young boundary as an empirical process and taking advantage of available concentration inequalities \cite{BLM}.

Fluctuations of random Young diagrams are studied in Section \ref{sec:5}, where we establish the pointwise asymptotic normality (Theorem \ref{th:CLT1}, unified by characterizing the limit as a Gaussian process in Theorems \ref{th:CLT2} and \ref{th:CLT3}. The case where convergence to the limit shape fails due to the number of sources $M$ not growing fast enough (referred to as a ``chaotic'' regime) is treated in Section \ref{sec:6}, where we demonstrate that a Poisson approximation is a suitable replacement of a deterministic limit (Theorems \ref{th:Poisson}, \ref{th:Poisson1} and \ref{th:Poisson2}), illustrated by a computer simulation. In Section \ref{sec:7}, the limit shape results are applied to some of the real-life frequency data sets considered earlier by Sichel \cite{Sichel1985} as a test bed for the proposed GIGP distribution goodness-of-fit. Here, our purpose it to demonstrate that the predicted limit shape, together with the adequate $(A,B)$-scaling  (based on estimated or predefined values of the GIGP parameters) provides a useful visualization tool for a quick and informative checkout of suitability of the GIGP model. The paper concludes with a summary of the main findings in Section \ref{sec:8}. Lastly, Appendix \ref{sec:A} comprises a brief compendium of asymptotic formulas for the Bessel function $K_\nu(z)$ involved in the definition of the GIGP distribution. These formulas (for small argument $z$ or for large order $\nu$) are being extensively used in our asymptotic analysis.

\section{Setting the scene}\label{sec:2}
\subsection{Items production model}\label{sec:2.1}

Suppose there are $M$ sources, each one producing a batch of items, and let $X_i$ denote the random size of the batch produced by the $i$-th source ($i=1,\dots,M$). The range of the output size can be $j\in\NN_0$ if empty output is allowed (e.g., citations of a paper), or it can be zero truncated, with $j\in\NN$ (e.g., papers of an author). The sources are independent of one another and their random outputs follow a common frequency distribution $(f_j)$, that is, the random variables $(X_i)$ are mutually independent and, for each $i=1,\dots,M$,
\begin{equation*}
\PP(X_i=j)=f_j\qquad (j\in\NN_0).
\end{equation*}
\begin{remark}
To streamline the notation, we keep writing $j\in\NN_0$, wherein the zero-truncated case is included with $f_0=0$.
\end{remark}
\begin{remark}\label{rm:occ}
The \emph{item production  model} introduced above can be rephrased as the classic \emph{occupancy problem}, dealing with independent allocation of $M$ particles over infinitely many boxes with probability distribution $(f_j)$  \cite{Gnedin}.
\end{remark}
We assume that the distribution $(f_j)$ has finite mean,
\begin{equation}\label{eq:eta0}
\eta:=\EE(X_i)=\sum_{j=0}^\infty j f_j<\infty.
\end{equation}
The total (random) number of produced items is given by the sum of the outputs,
\begin{equation}\label{eq:N}
N=\sum_{i=1}^M X_i,
\end{equation}
with the expected value
\begin{equation}\label{eq:EN}
\EE(N)=\sum_{i=1}^M \EE(X_i)=M\eta.
\end{equation}

It is useful to represent each $X_i$ via ``scanning'' across the range of possible values $j$,
\begin{equation}\label{eq:sum-chi}
X_i=\sum_{j=0}^\infty j\myp I_{\{X_i=j\}},
\end{equation}
where $I_{A}$ denotes the indicator of event $A$ (i.e., with values $1$ if $A$ occurs and $0$ otherwise). Of course,
\begin{equation}
\label{eq:EI}
\EE\bigl(I_{\{X_i=j\}}\bigr)=\PP(X_i=j)=f_j\qquad (j\in\NN_0).
\end{equation}
Consider the
multiplicity $M_j$ of output size $j\in\NN_0$ in the pooled production of items $(X_i)$,
\begin{equation}\label{eq:Mj}
M_j:=\#\bigl\{i\in\{1,\dots,M\}\colon X_i=j\bigr\}=\sum_{i=1}^M I_{\{X_i=j\}}\qquad (j\in\NN_0).
\end{equation}
Using \eqref{eq:EI}, we find the expectation
\begin{equation}\label{eq:=M(j)}
\EE(M_j)=\sum_{i=1}^M \EE\bigl(I_{\{X_i=j\}}\bigr)=Mf_j\qquad (j\in\NN_0).
\end{equation}
Note that the random variables $(M_j)$ are not independent; indeed, they sum up to the number of sources,
$$
\sum_{j=0}^\infty M_j=\sum_{j=0}^\infty \sum_{i=1}^M I_{\{X_i=j\}}=\sum_{i=1}^M \sum_{j=0}^\infty I_{\{X_i=j\}}=\sum_{i=1}^M 1=M.
$$

From the interpretation of the multiplicities $M_j$, it is evident that the total (random) number of produced items is given by
\begin{equation}\label{eq:N1}
N=\sum_{j=0}^\infty j\myp M_j.
\end{equation}
The same can be easily obtained using definition \eqref{eq:N} and decompositions \eqref{eq:sum-chi} and \eqref{eq:Mj},
$$
N=\sum_{i=1}^M X_i=
\sum_{i=1}^M \sum_{j=0}^\infty j\myp I_{\{X_i=j\}}=
\sum_{j=0}^\infty
j\sum_{i=1}^M I_{\{X_i=j\}} = \sum_{j=0}^\infty j\myp M_j.
$$
The expected value of $N$ can then be  expressed using
\eqref{eq:=M(j)} and \eqref{eq:eta0},
\begin{equation}\label{eq:EN1}
\EE(N)=\sum_{j=0}^\infty j\mypp\EE(M_j)=M\sum_{j=0}^\infty jf_j=M\eta,
\end{equation}
which is, of course, the same as \eqref{eq:EN}.

\begin{remark}\label{rm:eta-hat}
In view of formulas \eqref{eq:EN} and \eqref{eq:EN1}, the sample mean $\hat{\eta}=N/M$ is an unbiased estimator of the expected value $\eta$, possessing  all standard properties such as consistency and asymptotic normality. The advantage of this estimator is that it is \emph{non-parametric}, in the sense that it does not require knowledge of any distributional model $(f_j)$ behind the production output data.
\end{remark}

\subsection{Young diagrams and limit shape}\label{sec:2.2}
It is useful to rank the sources according to their production output, that is, by considering the (descending) order statistics $X_{1,M}\ge X_{2,M}\ge \dots\ge X_{M,M}$; for example, $X_{1,M}=\max_{1\le i\le M}\{X_i\}$ is the highest output score amongst $M$ sources. The production profile is succinctly visualized by the \emph{Young diagram} formed by the left- and bottom-aligned row blocks of unit height and lengths $X_{1,M}\ge X_{2,M}\ge \cdots$, respectively, with longer blocks positioned lower (see Fig.\mypp\ref{fig1}\myp(a)). In particular, blocks corresponding to the output value $j=0$ (if it is allowed) degenerate to vertical intervals (of height $1$ each) placed on top of the rest of the Young diagram along the vertical axis.

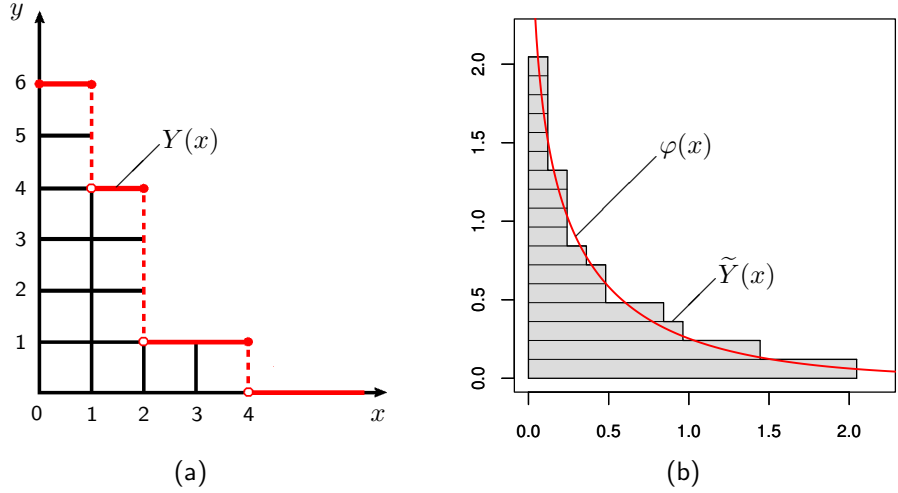
\begin{figure}[ht]
\centering
\mbox{}\hspace{.0pc}
\subfigure[]%
{\includegraphics[width=.407\textwidth]{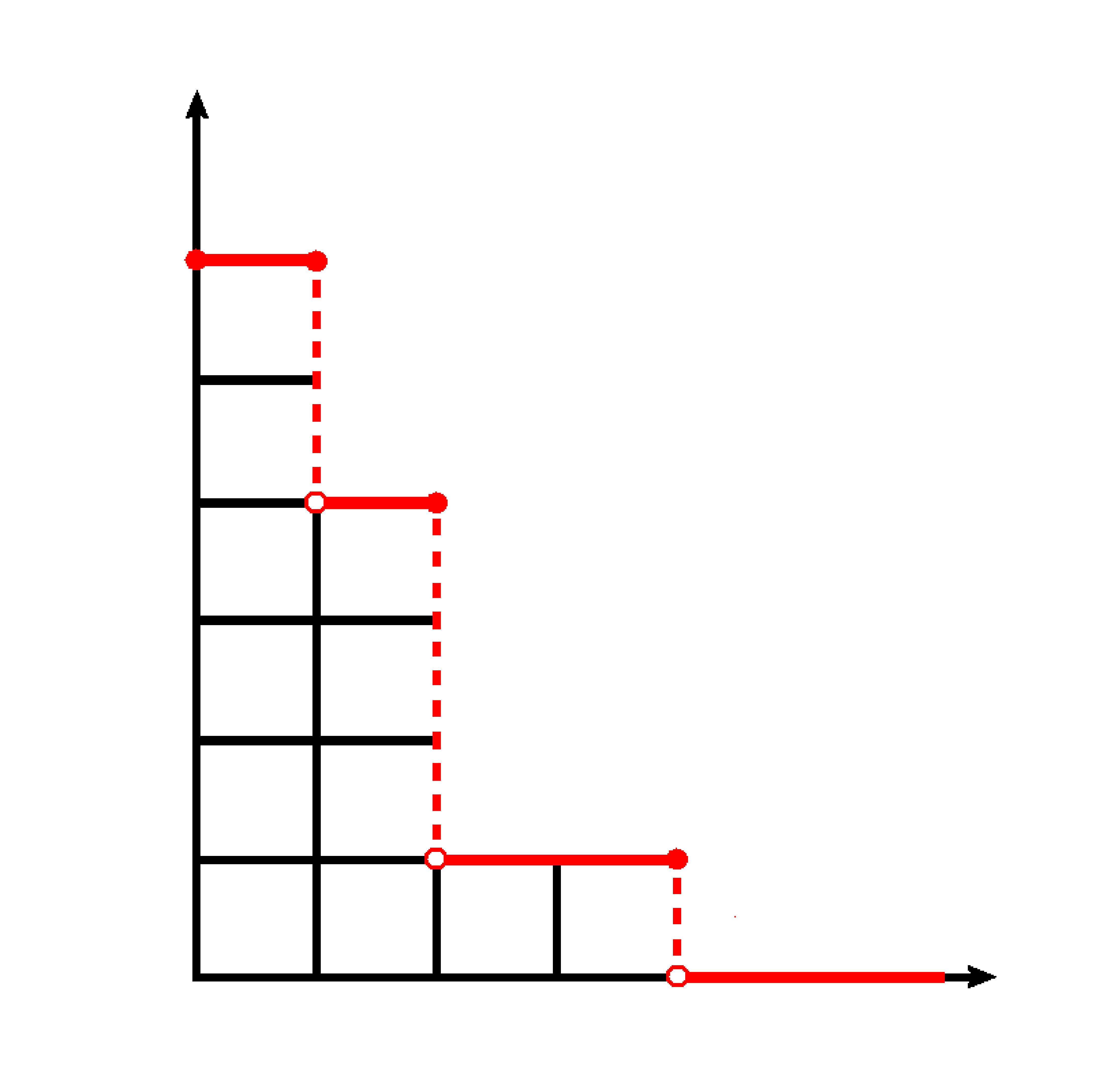}
\put(-25,8){\mbox{\footnotesize$x$}}
\put(-103,111){\mbox{\footnotesize$Y(x)$}}
\put(-104,112){\line(-1,-1){16}}
\put(-160,161){\mbox{\footnotesize$y$}}
\put(-152,8.3){\mbox{\scriptsize$\mathsf 0$}}
\put(-131.3,8){\mbox{\scriptsize$\mathsf 1$}}
\put(-112,8){\mbox{\scriptsize$\mathsf 2$}}
\put(-92,8){\mbox{\scriptsize$\mathsf 3$}}
\put(-72.6,8){\mbox{\scriptsize$\mathsf 4$}}
\put(-158,35){\mbox{\scriptsize$\mathsf 1$}}
\put(-158,55){\mbox{\scriptsize$\mathsf 2$}}
\put(-158,74){\mbox{\scriptsize$\mathsf 3$}}
\put(-158,93){\mbox{\scriptsize$\mathsf 4$}}
\put(-158,113){\mbox{\scriptsize$\mathsf 5$}}
\put(-158,133){\mbox{\scriptsize$\mathsf 6$}}
}\mbox{}\hspace{1pc}\subfigure[]{\includegraphics[width=.37\textwidth]{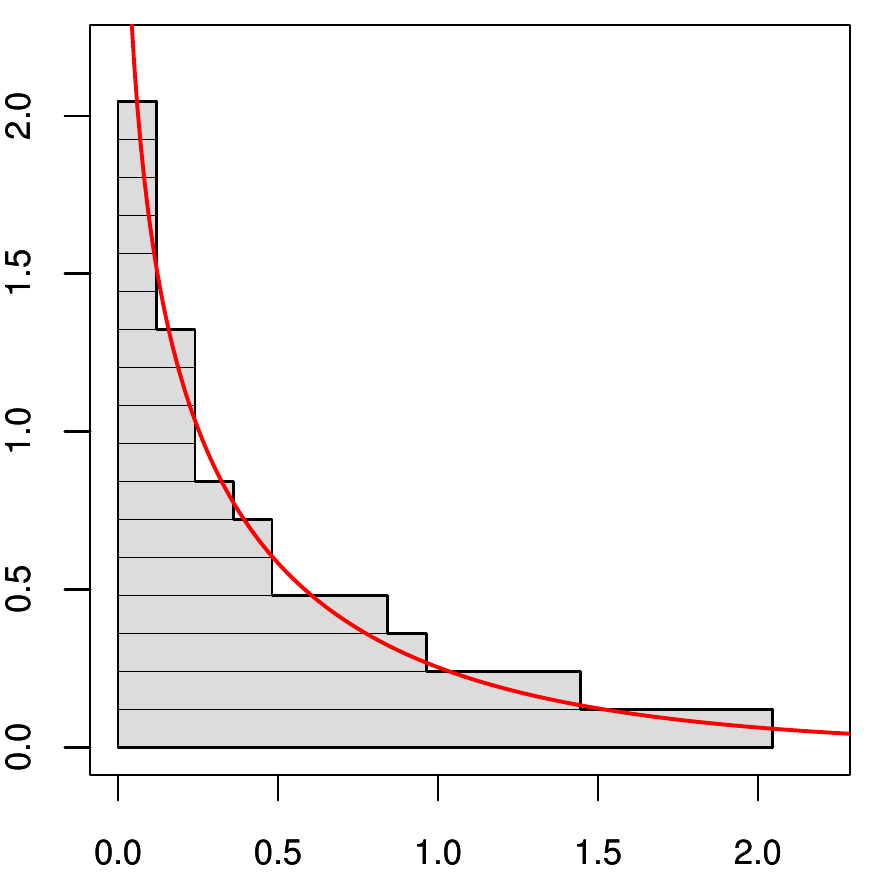}
\put(-71,60){\mbox{\footnotesize$\widetilde{Y}(x)$}}
\put(-72,61.5){\line(-1,-1){16}}
\put(-93,109){\mbox{\footnotesize$\varphi(x)$}}
\put(-94.2,107.8){\line(-1,-1){30}}
}
\caption{(a) Young diagram and the boundary $Y(x)$ for $M=6$ sources and ordered outputs $(X_{i,M})=(4,2,2,2,1,1)$, corresponding to counts $M_4=1$, $M_2=3$, $M_1=2$. (b) The limit shape $y=\varphi(x)$ for a randomized occupancy problem
(i.e., with independent counts $M_j$, see details in Section \ref{sec:2.3}), given by the equation $\rme^{-x\myp\pi/\sqrt{6}}+\rme^{-y\myp\pi/\sqrt{6}}=1$ (see \eqref{eq:LS02}). The shaded area represents a simulated Young diagram under the scaling $\sqrt{n}$ along both axes, with $n=100$; horizontal blocks correspond to the ordered outputs $X_{i,M}$ (with the sample value $M=17$).
} \label{fig1}
\end{figure}

The upper boundary of the Young diagram is the graph of the (left-continuous) step function
\begin{equation}\label{eq:Y}
    Y(x): = \sum_{j\ge x} \sum_{i=1}^M I_{\{X_i=j\}}=\sum_{j\ge x} M_j \qquad  (x\ge0)
\end{equation}
(see \eqref{eq:Mj}). To highlight the dependence on $x$, rewrite definition \eqref{eq:Y} in the form
\begin{equation}\label{eq:Y1}
    Y(x) =
    \sum_{j=0}^\infty M_j\mypp \mathbf{1}_{[0,j]}(x) \qquad  (x\ge0),
\end{equation}
where $\mathbf{1}_D(x)$ is the indicator function of set $D$ (i.e., $\mathbf{1}_D(x)=1$ if $x\in D$ and $\mathbf{1}_D(x)=0$ otherwise).

If $M_0>0$ then the function $Y(x)$ has an isolated peak at $x=0$; otherwise, $Y(x)$ is right-continuous at zero. The value at the origin is the total number of sources,
\begin{equation*}
Y(0)=\sum _{j=0}^\infty M_j=M,
\end{equation*}
whereas the area under the graph of $Y(x)$ equals the total number of produced items,
$$
\int_0^\infty \!Y(x)\,\rmd{x}=\sum_{j=0}^\infty M_j \int_0^\infty\!\mathbf{1}_{[0,j]}(x)\,\rmd{x}= \sum_{j=0}^\infty j\myp M_j=N
$$
(see \eqref{eq:Y1} and \eqref{eq:N1}).

Setting
\begin{equation}\label{eq:Zi}
Z_i(x):=\sum_{j\ge x}I_{\{X_i=j\}}=I_{\{X_i\ge x\}}\qquad (i=1,\dots,M),
\end{equation}
formula \eqref{eq:Y} can be expressed in the form
\begin{equation}\label{eq:YZ}
    Y(x) = \sum_{i=1}^M \sum_{j\ge x} I_{\{X_i=j\}}=\sum_{i=1}^M Z_i(x).
\end{equation}
The indicators $Z_1(x),\dots,Z_M(x)$ are independent and identically distributed Bernoulli random variables; specifically,
\begin{align}
\label{eq:Z1}
\PP(Z_i(x)=1)&=\PP(X_i\ge x)=\sum_{j\ge x} f_j =: \bar{F}(x),\\
\PP(Z_i(x)=0)&=\PP(X_i<x)=\sum_{j<x} f_j = :F(x),
\label{eq:Z0}
\end{align}
where $\bar{F}(x)+F(x)=1$ for all $x\ge0$. Hence,
$$
\EE\bigl(Z_i(x)\bigr)=\bar{F}(x),\qquad \Var\bigl(Z_i(x)\bigr)=\bar{F}(x)\myp\bigl(1-\bar{F}(x)\bigr)=\bar{F}(x)\myp F(x),
$$
and, for any $0\le x\le x'$,
$$
\Cov\bigl(Z_i(x),Z_{i}(x')\bigr)=\bar{F}(x') - \bar{F}(x)\mypp \bar{F}(x')=\bar{F}(x')\myp F(x).
$$
It then follows easily from \eqref{eq:YZ} that, for each $x\ge0$,
\begin{equation}\label{eq:Y-exp_var}
\EE \bigl(Y(x)\bigr)=M\bar{F}(x),\qquad \Var\bigl(Y(x)\bigr)=M\bar{F}(x)\myp F(x).
\end{equation}
and, for $0\le x\le x'$,
\begin{align}
\notag
 \Cov\bigl(Y(x),Y(x')\bigr)&= \sum_{i,\myp i'=1}^M\Cov\bigl(Z_i(x),Z_{i'}(x')\bigr)\\
&=\sum_{i=1}^M\Cov\bigl(Z_i( x),Z_{i}(x')\bigr)=M\bar{F}(x')\myp F(x) ,
\label{eq:Cov}
\end{align}

A useful visual insight into the structure of the production distribution may be obtained by looking at scaled Young diagrams, with some scaling coefficients $A$ and $B$,
\begin{equation}\label{eq:Y-tilde}
\widetilde{Y}(x)=\frac{1}{B}\,Y(A\myp x)=\frac{1}{B}\sum_{j\ge Ax} M_j=\frac{1}{B}\sum_{i=1}^MZ_i(A\myp x)\qquad (x\ge0).
\end{equation}
The aim is to seek a \emph{limit shape} $x\mapsto \varphi(x)$ such that, with suitable $A,B\to \infty$,
\begin{equation}\label{eq:exp-phi}
\EE \bigl(\widetilde{Y}(x)\bigr)\to\varphi(x)\qquad (x>0),
\end{equation}
and, moreover, there is convergence in probability of  $\widetilde{Y}(x)$ to $\varphi(x)$, that is, for any $\varepsilon>0$,
\begin{equation}\label{eq:p-phi}
\PP\bigl(|\widetilde{Y}(x)-\varphi(x)|>\varepsilon\bigr)\to0\qquad (x>0).
\end{equation}

\begin{remark}
The reason for restricting the range of convergence in \eqref{eq:exp-phi} and \eqref{eq:p-phi} to $x>0$ is that, in some cases, it turns out that $\varphi(0)=\infty$ (e.g., see Fig.\mypp\ref{fig1}\myp(b)).
\end{remark}

The notion of limit shape is motivated by similar topics in the theory of random integer partitions \cite{V0,vershik1996statistical}. This classic example is recalled briefly in Section \ref{sec:2.3} by way of illustration, although the setting there is somewhat different from the item production model. In the present paper, we address this problem for the GIGP distribution introduced in Section \ref{sec:3}.

\subsection{Example: limit shape of integer partitions}\label{sec:2.3}
To illustrate the concept of limit shape, we start with a baseline example of the power law frequency distribution, $f_j= j^{-a}\mynn/\zeta(a)$ ($j\ge 1$), with $a>1$. Choose any $A\to\infty$ such that  $B:=M/A^{a-1}\to\infty$; that is, $1\ll A\ll M^{1/(a-1)}$. Then the scaled expected Young diagram boundary function specializes to
\begin{align}
\label{eq:pl-sum-R}
\EE\bigl(\widetilde{Y}(x)\bigr)=\frac{A^{a-1}}{M}\!\sum_{j\ge Ax} Mj^{-a}
\mynn/\zeta(a)&=\frac{1}{\zeta(a)}\mynn\sum_{j/A\ge x} \left(\frac{j}{A}\right)^{-a}\frac{1}{A}\\
&\to \frac{1}{\zeta(a)}\mynn\int_x^\infty\! s^{-a}\,\rmd{s}=\frac{x^{-(a-1)}}{(a-1)\mypp\zeta(a)},
\label{eq:pl-sum}
\end{align}
using that the sum in \eqref{eq:pl-sum-R} is the Riemann integral sum of the integral in \eqref{eq:pl-sum}.
Thus, the limit shape exists and is given by the right-hand side of \eqref{eq:pl-sum}, but as mentioned in the Introduction, this is of no practical use because the scaling  parameter $A\to\infty$ is arbitrary as long as $A=o(M^{1/(a-1)})$ (which confirms that the power law distribution is \emph{scale free}).

The classic example of a frequency model possessing a meaningful limit shape comes from the theory of random integer partitions. Here, the values $j=1,2,\dots$ are interpreted as candidate parts into an integer partition, and the corresponding multiplicity $M_j$ is the number of times the part $j$ is used, respectively. In particular, if $M_j=0$ then the value $j$ is not involved in the partition, and it is tacitly assumed that only finitely many of $M_j$'s are non-zero. The sum $N=\sum_{j=1}^\infty j\myp M_j$ yields the integer being partitioned into the sum of the parts $j$ with $M_j>0$.

The standard model set-up there is different from the item production model described in Section \ref{sec:2.1}. Namely, instead of the premise of $M$ independent sources, with multiplicities $(M_j)$ expressed by formula \eqref{eq:Mj}, the randomized partition model is defined by assuming that the multiplicities $(M_j)$ are independent random variables with geometric distribution, $M_j\sim \mathrm{Geom}\myp(1-z^j)$ ($j\ge1$), that is,
\begin{equation}\label{eq:geom}
\PP(M_j=m)=z^{j\myp m}\mypp(1-z^j)\qquad (m\ge 0),
\end{equation}
with the expected value given by
\begin{equation}\label{eq:exp-geom}
\EE(M_j)=\frac{z^{j}}{1-z^j}\qquad (j\ge 1).
\end{equation}
The parameter $z\in(0,1)$ is chosen specifically as
\begin{equation}\label{eq:zn}
z=\rme^{-\kappa/\sqrt{n}},\qquad \kappa:=\frac{\pi}{\sqrt{6}}=\sqrt{\zeta(2)},
\end{equation}
where $n$ is an external (large) parameter.

Note that, for any $z\in(0,1)$,
$$
\PP(M_j>0)=1-\PP(M_j=0)=1- (1-z^j)=z^j,
$$
and
$$
\sum_{j=1}^\infty \PP(M_j>0)=\sum_{j=1}^\infty z^j=\frac{z}{1-z}<\infty.
$$
Therefore, by the Borel--Cantelli lemma (see, e.g., \cite[Sec.\,II.10, p.\,255]{Shiryaev}), the number of nonzero terms in the sequence of random multiplicities $(M_j)$ is finite with probability $1$.

Due to the mutual independence of $M_j$ and the geometric marginal distributions \eqref{eq:geom}, the probability of a given sequence of multiplicities $M_j=m_j$ ($j\ge 1$) (with finitely many nonzero terms) is expressed as follows,
\begin{equation}\label{eq:Boltzmann}
\PP(M_j=m_j, \,j=1,2,\dots)=\prod_{j=1}^\infty
z^{j\myp m_j}(1-z^j)=\frac{z^{N}}{G(z)},
\end{equation}
where
$N=\sum_{j=1}^\infty j\mypp m_j$
and
$$
G(z)=\prod_{j=1}^\infty \frac{1}{1-z^j}\qquad (0<z<1).
$$
Formula \eqref{eq:Boltzmann} is an instance of the so-called \emph{Boltzmann distribution}, with roots in statistical physics \cite{Auluck,Vershik2} and many applications in probabilistic combinatorics \cite{ABT} and computing \cite{Duchon}.

Motivation for the choice of the Boltzmann distribution
\eqref{eq:Boltzmann} is due to the fact that its conditioning leads to the uniform distribution on the corresponding subspace. Specifically, denoting by $\varPi_n$ the set of all integer partitions of $n$, it is easy to see that
the conditional probability of any partition in $\varPi_n$ with specific multiplicities of parts $M_j=m_j$, conditioned on $N=\sum_{j=1}^\infty jM_j=n$, is given by
$$
\PP\bigl(M_j=m_j,\, j\ge1\,\big|\,N={\textstyle{\sum_{j}}}\mypp j\myp M_j=n\bigr)=\frac{z^n/G(z)}{(z^n/G(z)) \cdot \#\varPi_n}=\frac{1}{\#\varPi_n},
$$
which is the uniform distribution on $\varPi_n$. Furthermore, the choice of the parameter $z$ in the asymptotic form \eqref{eq:zn} is explained by the natural calibration condition
\begin{equation}\label{eq:exp=N}
\EE(N)=\EE\Bigl({\textstyle \sum_{j=1}^\infty}\mypp j\myp M_j\Bigr)\sim n\qquad (n\to\infty).
\end{equation}
Indeed, using the mean formula \eqref{eq:exp-geom} and seeking the parameter $z$ in the form $z=\rme^{-\alpha_n}$, with $\alpha_n\to0$, the asymptotic equation \eqref{eq:exp=N} is rewritten as
\begin{equation}\label{eq:exp-N-Riemann}
\EE(N)=\sum_{j=1}^\infty \frac{j\,\rme^{-\alpha_n j}}{1-\rme^{-\alpha_n j}}=\frac{1}{\alpha_n^2}\sum_{j=1}^\infty \frac{\alpha_n j\,\rme^{-\alpha_n j}}{1-\rme^{-\alpha_n j}}\,\alpha_n\sim n.
\end{equation}
Observing that the sum in \eqref{eq:exp-N-Riemann} is a Riemann integral sum, it follows that
\begin{align*}
\sum_{j=1}^\infty \frac{\alpha_n j\,\rme^{-\alpha_n j}}{1-\rme^{-\alpha_n j}}\, \alpha_n&\to\int_0^\infty\!\frac{s\,\rme^{-s}}{1-\rme^{-s}}\,\rmd{s}\\
&=\sum_{\ell=1}^\infty\int_0^\infty\! s\,\rme^{-\ell s}\,\rmd{s}
=\sum_{\ell=1}^\infty\frac{1}{\ell^2}=\zeta(2)=\frac{\pi^2}{6}=\kappa^2.
\end{align*}
Substituting this into equation \eqref{eq:exp-N-Riemann}, we obtain $\alpha_n\sim \kappa/\sqrt{n}$, in line with \eqref{eq:zn}.

The expected limit shape in the partition model can now be easily computed \cite{vershik1996statistical,bogachev2015unified}: setting $A=B=\sqrt{n}$, we have, for any $x>0$,
\begin{align}
\notag
\EE\bigl(\widetilde{Y}(x)\bigr)=\frac{1}{B}\sum_{j\ge Ax}\EE(M_j)&=\frac{1}{\sqrt{n}}\sum_{j\ge \sqrt{n}\,x}
\frac{\rme^{-\alpha_n j}}{1-\rme^{-\alpha_n j}}\\
\notag
&\to\frac{1}{\kappa}\int_{\kappa x}^\infty \!\frac{\rme^{- u}}{1-\rme^{u}}\,\rmd{u}=\frac{1}{\kappa}\sum_{\ell=1}^\infty\int_{\kappa x}^\infty\!  \rme^{-\ell u}\,\rmd{u}\\
&=\frac{1}{\kappa}\sum_{\ell=1}^\infty \frac{1}{\ell}\,\rme^{-\ell x}=-\frac{1}{\kappa}\log \myn(1-\rme^{-\kappa x}).
\label{eq:phi}
\end{align}
Thus, the limit shape $y=\varphi(x)$ is given by the equation
\begin{equation*}
y=-\kappa^{-1} \log \myn(1-\rme^{-\kappa x})\qquad (x>0),
\end{equation*}
or, in a more symmetric form,
\begin{equation}\label{eq:LS02}
\rme^{-\kappa x}+\rme^{-\kappa y}=1\qquad (x, y > 0),
\end{equation}
where $\kappa=\pi/\sqrt{6}$ (see \eqref{eq:zn}). The plot of this function is shown in Fig.\mypp\ref{fig1}\myp(b).

Note that $\varphi(0)=\infty$. According to the calculation in \eqref{eq:phi}, this implies that the expected value of $M$ grows faster than $\sqrt{n}$. More precisely, we have
\begin{equation}\label{eq:exp-M-Riemann}
\EE(M)=\sum_{j=1}^\infty \frac{\rme^{-\alpha_n j}}{1-\rme^{-\alpha_n j}}=\sum_{j=1}^m\frac{\rme^{-\alpha_n j}}{1-\rme^{-\alpha_n j}}+\frac{1}{\alpha_n}\sum_{j>m}^\infty \frac{\alpha_n\mypp\rme^{-\alpha_n j}}{1-\rme^{-\alpha_n j}},
\end{equation}
where $m=[1/\alpha_n]$ and $\alpha_n=-\kappa/\sqrt{n}$ (see \eqref{eq:zn}). Arguing as before, we see that the last sum in \eqref{eq:exp-M-Riemann} converges to the integral $\int_1^\infty \rme^{-\kappa u}\mypp(1-\rme^{-\kappa u})^{-1}\mypp\rmd{u}<\infty$. Next, write
$$
\sum_{j=1}^m\frac{\rme^{-\alpha_n j}}{1-\rme^{-\alpha_n j}}=\frac{1}{\alpha_n}\sum_{j=1}^m\frac{1}{j}+\sum_{j=1}^m\left(\frac{\rme^{-\alpha_n j}}{1-\rme^{-\alpha_n j}}-\frac{1}{\alpha_n j}\right)\!,
$$
where \cite[2.10.8]{NIST}
$$
\sum_{j=1}^m\frac{1}{j}\sim \log m\sim -\log\alpha_n
$$
and
$$
\sum_{j=1}^m\left(\frac{\rme^{-\alpha_n j}}{1-\rme^{-\alpha_n j}}-\frac{1}{\alpha_n j}\right)\sim\frac{1}{\alpha_n}\int_0^1 \!\left(\frac{\rme^{-u}}{1-\rme^{-u}}-\frac{1}{u}\right)\myn\rmd{u}=O(\alpha_n^{-1}),
$$
noting that the integrand function has a finite limit at zero.
As a result,
\begin{equation}\label{eq:M>>}
\EE(M)\sim \alpha_n^{-1} (-\log\alpha_n)\sim   \frac{\sqrt{n}}{2\myp\kappa}\myp \log n=\frac{\sqrt{6\myp n}}{2\pi}\myp\log n\qquad (n\to\infty).
\end{equation}

\begin{remark}
Two different model settings discussed above\,---\,with independent outputs $X_i$ ($i=1,\dots, M$), like in the item production model (Section \ref{sec:2.2}), or with independent miltiplicities $M_j$ ($j\in\NN_0$), like in a randomized model of integer partitions (Section \ref{sec:2.3}), are in fact closely connected  and, in a sense, equivalent to one another. Indeed, randomization of certain parameters in combinatorial structures is a frequently used technical tool \cite{ABT} aiming to overcome structural constraints, such as a prescribed sum of parts in integer partitions \cite{Fristedt,bogachev2015unified}. As another example directly related to the item production model, in the occupancy problem (see Remark \ref{rm:occ}) it is conventional to use the so-called \emph{poissonization} \cite{ABT, Borisov} by replacing the original (co-dependent) multiplicities $M_j$ by independent Poisson random variables with mean $Mf_j$, respectively ($j\in\NN_0$) \cite{BGY,Gnedin}. In each of these settings, the anticipated equivalence is guaranteed via a suitable ``bridge'' between the original and randomized versions of the problem, such as a local limit theorem for the asymptotics of probabilities $\PP\bigl(\sum_j jM_j=n\bigr)$ in the case of integer partitions \cite{Fristedt,bogachev2015unified}, or a ``depoissonization lemma'' in the occupancy problem \cite{Gnedin, BGY}.
\end{remark}

\section{The GIGP model}\label{sec:3}
\subsection{The GIGP distribution}\label{sec:3.1}
The \emph{generalized inverse Gaussian-Poisson (GIGP)} distribution
introduced by Sichel \cite{Sichel1971, Sichel1985} is of the form
\begin{equation}\label{eq:GIGP}
    f_j=\frac{\left(1-\theta\right)^{\nu/2}}{K_\nu\bigl(\alpha\myn
    \left(1-\theta\right)^{1/2}\bigr)} \cdot\frac{\left(\frac12\myp\alpha\mypp \theta\right)^j}{j!}\mypp K_{\nu+j}(\alpha)\qquad (j\in\NN_0),
\end{equation}
where parameters have the range $\nu\in\RR$, $\alpha>0$ and $0<\theta<1$, and $K_\nu(\cdot)$ is the \emph{modified Bessel function of the second
kind} of order $\nu$ \cite[\S\myp10.25(i), \S\myp10.25(ii)]{NIST}.

As was mentioned in the Introduction, the GIGP model \eqref{eq:GIGP}  is a mixed Poisson distribution,
\begin{equation}\label{eq:GIPG1}
f_j=\int_0^\infty \frac{\lambda^j\myp\rme^{-\lambda}}{j!}\,g(\lambda)\,\rmd{\lambda}\qquad (j\ge 0),
\end{equation}
 with the mixing density for the Poisson parameter $\lambda$ chosen as a \emph{generalized inverse Gaussian (GIG)}
density \cite{Sichel1971} (see also \cite[p.\:284]{Johnson})\footnote{We follow the nomenclature of  \cite{Sichel1971}. The connection with an alternative parameterization $(\theta,\psi,\chi)$ in \cite{Johnson} is via the maps $\theta\mapsto\nu$, $\psi\mapsto 2\myp(1-\theta)/\theta$, $\chi\mapsto \alpha^2\myp\theta/2$.}
\begin{equation}\label{eq:g}
g(\lambda)=\frac{\bigl(2\left(1-\theta\right)^{1/2}\!\mynn/\alpha\theta\bigr)^\nu}{2\,K_\nu\bigl(\alpha\left(1-\theta\right)^{1/2}\bigr)}\, \lambda^{\nu-1}\exp{}\!\!\left(-\frac{(1-\theta)\,\lambda}{\theta}-\frac{\alpha^2\theta}{4\myp\lambda}\right)\qquad(\lambda>0).
\end{equation}
The normalization in \eqref{eq:g} is due to one of the integral representations for the Bessel function \cite[10.32.10]{NIST}. Representation \eqref{eq:GIPG1} explains why formula \eqref{eq:GIGP} defines a probability distribution,
\begin{align*}
\sum_{j=0}^\infty f_j=\int_0^\infty \sum_{j=0}^\infty \frac{\lambda^j\myp\rme^{-\lambda}}{j!}\,g(\lambda)\,\rmd{\lambda}= \int_0^\infty \!g(\lambda)\,\rmd{\lambda}=1,
\end{align*}
and it also leads to a curious identity for the Bessel functions, which does not seem to have been mentioned in the special functions literature,
\begin{equation}\label{eq:==K}
\sum_{j=0}^\infty \frac{\left(\frac12\myp\alpha\mypp \theta\right)^j\mynn K_{\nu+j}(\alpha)}{j!} =\frac{K_\nu\bigl(\alpha\myn
    \left(1-\theta\right)^{1/2}\bigr)}{\left(1-\theta\right)^{\nu/2}}.
\end{equation}

From formula \eqref{eq:GIPG1}, the \strut{}expression \eqref{eq:GIGP} is easily obtained using the normalization of the GIG density \eqref{eq:g} with parameters $\theta$ and $\alpha$ replaced by $\tilde{\theta}=\theta/(1+\theta)$ and $\tilde{\alpha}=\alpha\,\sqrt{1+\theta\mypp}$, respectively.
Furthermore, formula \eqref{eq:GIPG1} implies that the expected value of the GIGP distribution \eqref{eq:GIGP} coincides with that of the 
GIG distribution \eqref{eq:g},
\begin{align}
\notag
\eta=\sum_{j=0}^\infty  jf_j=\int_0^\infty \mynn \sum_{j=0}^\infty j\,\frac{\lambda^j\myp\rme^{-\lambda}}{j!}\,g(\lambda)\,\rmd{\lambda}&= \int_0^\infty \!\lambda\,g(\lambda)\,\rmd{\lambda}\\
&=\frac{\alpha\mypp\theta}{2\left(1-\theta\right)^{1/2}}\cdot\frac{K_{\nu+1}\bigl(\alpha\left(1-\theta\right)^{1/2}\bigr)}{K_{\nu}\bigl(\alpha\left(1-\theta\right)^{1/2}\bigr)},
\label{eq:eta}
\end{align}
where the last computation is based on the normalization in \eqref{eq:g} with order $\nu+1$.\footnote{Expression \eqref{eq:eta} follows directly from the definition \eqref{eq:GIGP} by using the identity \eqref{eq:==K} with order $\nu+1$.}

As was pointed out by Sichel \cite[p.\:315]{Sichel1985}, the frequencies \eqref{eq:GIGP} satisfy the recurrence relation
\begin{equation*}
f_{j+2}=\frac{(\nu+j+1)\mypp\theta}{j+2}\,f_{j+1}+\frac{\alpha^2\myp\theta^2}{4\left(j+2\right)\left(j+1\right)}\,f_{j}\qquad
(j\in\NN_0),
\end{equation*}
which can be obtained by integration by parts of the integral representation mentioned above after formula \eqref{eq:g}.

The tail of the GIGP distribution \eqref{eq:GIGP} has a power-geometric decay, as can be shown using Stirling's formula \cite[5.11.3]{NIST} and the asymptotics \eqref{eq:K5} of the Bessel function of large order, yielding
\begin{equation}\label{eq:Kj1}
 f_j\sim\frac{\left(1-\theta\right)^{\nu/2}\bigl(\tfrac12\myp \alpha\bigr)^{-\nu}}{2\myp K_\nu\bigl(\alpha\myn
    \left(1-\theta\right)^{1/2}\bigr)}\mypp  j^{\nu-1}\myp \theta^j\qquad(j\to\infty).
\end{equation}

\subsection{The boundary case \texorpdfstring{$\alpha=0$}{alpha}}
\label{sec:3.2}
The value $\alpha=0$ can also be included in the GIGP class via the limit $\alpha\to0+$. To this end, we need to consider several cases for the value of the order $\nu$. Namely, if $\nu>0$ then, using the small argument asymptotics of the Bessel function (see \eqref{eq:K2}), we obtain from \eqref{eq:GIGP}
\begin{equation}\label{eq:f01}
f_j\sim \left(1-\theta\right)^\nu\,\frac{\Gamma(\nu+j)\,\theta^j}{\Gamma(\nu)\,j!}=\binom{\nu+j-1}{j}\left(1-\theta\right)^\nu\theta^j \qquad(j\in\NN_0),
\end{equation}
where $\Gamma(z):=\int_0^\infty s^{z-1}\mypp \rme^{-s}\,\rmd{s}$ ($z>0$) is the gamma function \cite[5.2.1]{NIST}. Formula \eqref{eq:f01} defines a \emph{negative binomial distribution} with parameters $\nu$ and $\theta$ \cite[Sec.\:5.1]{Johnson1}, with  the expected value given by
\begin{equation}\label{eq:ex01}
\eta=\frac{\nu\mypp \theta}{1-\theta}.
\end{equation}
The latter  expression is consistent with the limit of \eqref{eq:eta} as $\alpha\to0+$ (again using \eqref{eq:K2}). The tail behavior of \eqref{eq:f01} is retrieved with the aid of Stirling's formula \cite[5.11.3]{NIST},
\begin{equation}\label{eq:Kj2}
 f_j\sim\frac{\left(1-\theta\right)^{\nu}  \mynn j^{\nu-1}\myp \theta^j}{\Gamma(\nu)}\qquad(j\to\infty),
\end{equation}
which is formally in agreement with the
limit of \eqref{eq:Kj1} as $\alpha\to0+$.

However, for $\nu\le0$ the limiting GIGP distribution  degenerates to $f_0=1$ and  $f_j=0$ for all $j\ge1$. Indeed, for $\nu=0$ we get, using the asymptotic formula \eqref{eq:K3},
$$
f_0=\frac{K_0(\alpha)}{K_0\bigl(\alpha\left(1-\theta\right)^{1/2}\bigr)}\sim\frac{-\log \alpha}{-\log\bigl(\alpha\left(1-\theta\right)^{1/2}\bigr)}\to1.
$$
For $\nu<0$, with the aid of the asymptotic formulas \eqref{eq:K1} and \eqref{eq:K2}
we have
$$
f_0=\frac{\left(1-\theta\right)^{\nu/2} K_\nu(\alpha)}{K_\nu\bigl(\alpha\left(1-\theta\right)^{1/2}\bigr)}\sim\frac{\left(1-\theta\right)^{\nu/2}\frac12\mypp\Gamma(-\nu)\bigl(\frac12\myp\alpha\bigr)^{\nu}}{\frac12\mypp\Gamma(-\nu)\bigl(\frac12\myp\alpha\left(1-\theta\right)^{1/2}\bigr)^{\nu}}=1.
$$

To rectify this degeneracy, we switch to the zero-truncated GIGP distribution defined by
$$
\PP(X_i=j\,|\mypp X_i\ge 1) =\frac{f_j}{1-f_0}\qquad (j\in\NN)
$$
and taken in the limit as $\alpha\to0+$. We denote the resulting conditional frequencies by $(\check{f}_j)$ ($j\in\NN$), and the corresponding expected value by $\check{\eta}$. We restrict analysis to the range $-1<\nu\le 0$, and consider separately the cases $\nu=0$ and $-1<\nu<0$
(see Remark \ref{rm:nu=-1} below for why the value $\nu=-1$ is not compatible with $\alpha=0$).

\begin{remark}
The case $\nu<-1$ with $\alpha>0$ is excluded from consideration (see Proposition \ref{pr:eta}\myp(e) and a comment before this proposition). Hence, it is of no interest for us to consider the limit $\alpha\to0$ here.
\end{remark}

\noindent \underline{Case $\nu=0$}

\medskip
\noindent Applying the asymptotic formula \eqref{eq:K3}, we obtain
$$
1-f_0=\frac{K_0\bigl(\alpha\left(1-\theta\right)^{1/2}\bigr)-K_0(\alpha)}{K_0\bigl(\alpha\left(1-\theta\right)^{1/2}\bigr)}\sim \frac{\log\myn(1-\theta)}{\log \alpha},
$$
whereas \eqref{eq:K2} and \eqref{eq:K30} give for $j\ge1$
$$
f_j=\frac{\left(\frac12\myp\alpha\mypp \theta\right)^j}{j!}\cdot \frac{K_{j}(\alpha)}{K_0\bigl(\alpha\myn
    \left(1-\theta\right)^{1/2}\bigr)} \sim \frac{1}{-\log \alpha }\cdot \frac{\theta^j}{j},
    $$
using that $\Gamma(j)=(j-1)!$. Hence,
\begin{equation}\label{eq:f02}
\frac{f_j}{1-f_0}\sim \check{f}_j:=\frac{1}{-\log\left(1-\theta\right)}\cdot\frac{\theta^j}{j}\qquad (j\in\NN),
\end{equation}
which is \emph{Fisher's logarithmic series distribution} \cite[Sec.\,7.1.2]{Johnson1}. Note that the tail behavior of \eqref{eq:f02} is automatically power-geometric akin to \eqref{eq:Kj2} (with $\nu=0$).
The expected value of this distribution is easily computed,
\begin{equation}\label{eq:ex02}
 \check{\eta}=\frac{1}{-\log \left(1-\theta\right)}\sum_{j=1}^\infty \theta^j=\frac{\theta}{\left(1-\theta\right)\mynn\bigl(-\log \left(1-\theta\right)\bigr)}.
\end{equation}

\noindent\underline{Case $-1<\nu<0$}

\medskip
\noindent With the aid of the asymptotic formula \eqref{eq:K4} we get
\begin{align*}
1-f_0&=
    \frac{K_\nu\bigl(\alpha\myn
    \left(1-\theta\right)^{1/2}\bigr)-\left(1-\theta\right)^{\nu/2}K_{\nu}(\alpha)}{K_\nu\bigl(\alpha\myn
    \left(1-\theta\right)^{1/2}\bigr)}\\
    &\sim \frac{\Gamma(\nu+1)}{(-\nu)\,\Gamma(-\nu)}\,\bigl(\tfrac12\myp\alpha
   \bigr)^{-2\myp\nu} \left(1-\left(1-\theta\right)^{-\nu}\right),
\end{align*}
and furthermore, for $j\ge1$,
$$
f_j\sim \frac{ \left(1-\theta\right)^{\nu/2}\left(\frac12\myp\alpha\mypp \theta\right)^j}{j!}\cdot \frac{\frac12\mypp\Gamma(\nu+j)\left(\frac12\myp\alpha\right)^{-\nu-j}}{\frac12\mypp\Gamma(-\nu)\mypp\bigl(\frac12\myp\alpha\myn
    \left(1-\theta\right)^{1/2}\bigr)^\nu} \sim \frac{\Gamma(\nu+j)\,\theta^j}{\Gamma(-\nu)\,j!}\left(\tfrac12\myp\alpha\right)^{-2\myp\nu}.
$$
Hence,
\begin{equation}\label{eq:ext-heg-bin}
\frac{f_j}{1-f_0}\sim \check{f}_j:=\frac{(-\nu)\,\Gamma(\nu+j)\,\theta^j}{\Gamma(\nu+1)\myn\left(1-\left(1-\theta\right)^{-\nu}\right)\myn j!}\qquad (j\in\NN).
\end{equation}
This is an \emph{extended negative binomial} distribution \cite[Sec.\,5.12.2]{Johnson1}, with the expected value
\begin{equation}\label{eq:ex03}
\check{\eta}=\frac{(-\nu)\,\theta\left(1-\theta\right)^{-\nu-1}}{1-\left(1-\theta\right)^{-\nu}}.
\end{equation}
The tail decay of the distribution \eqref{eq:ext-heg-bin} is easily obtained using Stirling's formula \cite[5.11.3]{NIST},
\begin{equation}\label{eq:Kj3}
\check{f}_j\sim \frac{(-\nu)\mypp j^{\nu-1}\mypp\theta^j{}}{\Gamma(\nu+1)\left(1-\left(1-\theta\right)^{-\nu}\right)}\qquad(j\to\infty).
\end{equation}

\begin{remark}\label{rm:nu=-1}
If $\nu=-1$ then, using \eqref{eq:K1}, \eqref{eq:K31} and \eqref{eq:K3}, we have
\begin{align*}
1-f_0
&=
\frac{K_1\bigl(\alpha
   \myn \left(1-\theta\right)^{1/2}\bigr)-\left(1-\theta\right)^{-1/2}K_{1}(\alpha)}{K_1\bigl(\alpha\myn
    \left(1-\theta\right)^{1/2}\bigr)}\sim \tfrac12\myp\alpha^2\myp\theta\left(-\log\alpha\right)\mynn,
\end{align*}
and
\begin{equation*}
f_1=\frac{\left(1-\theta\right)^{-1/2}\bigl(\tfrac12\myp\alpha\mypp \theta\bigr)\mypp K_{0}(\alpha)}{K_1\bigl(\alpha\myn
    \left(1-\theta\right)^{1/2}\bigr)}\sim \tfrac12\myp\alpha^2\myp\theta\left(-\log\alpha\right),
\end{equation*}
hence
$$
\frac{f_1}{1-f_0}\sim \check{f}_1=1.
$$
Thus, the limiting conditional distribution $(\check{f}_j)$ appears to be degenerate, with all mass concentrated at $j=1$. This is unsuitable for the modeling purposes, which explains why  the ``corner'' case $\nu=-1$, $\alpha=0$ is excluded from consideration.
\end{remark}

\subsection{Asymptotics of the GIGP mean}\label{sec:3.3}

As indicated by the integer partition example in Section \ref{sec:2.2}, for the existence of a meaningful limit shape, the area of the Young diagram must grow faster than the number of constituent blocks (see \eqref{eq:M>>}). In the context of the item production model, this means that the total number of items, $N=\sum_j j\myp M_j$, should be much larger than the number of sources, $M=\sum_j M_j$. Recalling from \eqref{eq:EN} that the expected total number of items is given by $\EE(N)=M\eta$ (where \strut{}$\eta=\EE(X_i)$ is the expected number of items per source, see \eqref{eq:eta0}), this implies that a suitable limiting regime is determined by $\eta\to\infty$.

In turn, from the expression \eqref{eq:eta} for the GIGP mean $\eta$, one can hypothesize that the latter is achieved if $\theta\approx 1$, while the parameters $\alpha$ and $\nu$ are kept fixed. This can be verified (cf.\ Proposition \ref{pr:eta} below) using the known asymptotic formulas for the Bessel function $K_\nu(z)$ with $z\to0$, adapted to our needs in the next lemma.
\begin{lemma}\label{lm:K}
For $\alpha>0$ and $\nu\in\RR$ fixed, the following asymptotics hold as\/ $\theta\to1-$,
\begin{equation}
\label{eq:KK}
K_\nu\bigl(\alpha\left(1-\theta\right)^{1/2}\bigr)\sim \begin{cases}
\tfrac12\mypp\Gamma(\nu)\bigl(\tfrac12\myp\alpha\bigr)^{-\nu}\mynn\left(1-\theta\right)^{-\nu/2}&(\nu>0),\\[.4pc]
\tfrac12\bigl(-\log\myn(1-\theta)\bigr)&(\nu=0),\\[.4pc]
\tfrac12\mypp\Gamma(-\nu)\bigl(\tfrac12\myp\alpha\bigr)^{\nu}\mynn\left(1-\theta\right)^{\nu/2}&(\nu<0).
\end{cases}
\end{equation}
\end{lemma}
\begin{proof}
 The leading terms of the asymptotics \eqref{eq:KK} follow directly from formulas \eqref{eq:K2} for $\nu\ne 0$ (with the aid of \eqref{eq:K1} for $\nu<0$) and \eqref{eq:K30} for $\nu=0$,
\end{proof}

Using this lemma, we can  characterize more precisely the asymptotic behavior of the GIGP mean  in the limit as $1-\theta\to 0+$. In particular, this analysis reveals that the desired growth to infinity is in place for $\nu\ge-1$, but fails for $\nu<-1$.
\begin{proposition}\label{pr:eta}
The expected values $\eta$ and $\check{\eta}$
of the GIGP $(\alpha>0)$ and zero-truncated GIGP $(\alpha=0)$ distributions, respectively,
have the following asymptotics as $\theta\to1-$.
\begin{itemize}
\item[\rm(a)]
$\nu>0$, \,$\alpha\ge 0$\mypp\textup{:}
\begin{equation}\label{eq:eta1}
\eta\sim\frac{\nu}{1-\theta}.
\end{equation}
\item[\rm(b)]
$\nu=0$, \,$\alpha\ge0$\mypp\textup{:}
\begin{equation}\label{eq:eta2}
 \eta\sim \frac{1}{(1-\theta)\myp\bigl(-\log\myn(1-\theta)\bigr)},\qquad  \check{\eta}\sim \frac{1}{(1-\theta)\myp\bigl(-\log\myn(1-\theta)\bigr)}.
\end{equation}
\item[\rm(c)]
$-1<\nu<0$, \,$\alpha\ge0$\mypp\textup{:}
\begin{equation}\label{eq:eta3}
\eta\sim\frac{\Gamma(\nu+1)\bigl(\tfrac12\myp\alpha\bigr)^{-2\myp\nu}}{\Gamma(-\nu)\left(1-\theta\right)^{\nu+1}},\qquad \check{\eta}\sim\frac{-\nu}{\left(1-\theta\right)^{\nu+1}}.
\end{equation}

\item[\rm(d)]
$\nu=-1$, \,$\alpha>0$\mypp\textup{:}
\begin{equation}\label{eq:eta4}
\eta\sim \bigl(\tfrac12\myp\alpha\bigr)^2 \bigl(-\log\myn(1-\theta)\bigr).
\end{equation}

\item[\rm(e)]
$\nu<-1$, \,$\alpha>0$\mypp\textup{:}
\begin{equation}\label{eq:eta5}
\eta\sim\frac{\bigl(\tfrac12\myp\alpha\bigr)^{2}}{-\nu-1}.
\end{equation}

\end{itemize}

\end{proposition}
\begin{proof} Consider cases (a)--(e) using the asymptotic formulas of Lemma \ref{lm:K}.
\begin{itemize}
\item[\rm (a)] For $\alpha>0$,
using the first line of \eqref{eq:KK} for orders $\nu$ and $\nu+1$, we have
\begin{equation}\label{eq:eta1a}
\frac{K_{\nu+1}\bigl(\alpha\left(1-\theta\right)^{1/2}\bigr)}{K_{\nu}\bigl(\alpha\left(1-\theta\right)^{1/2}\bigr)}\sim \frac{\frac12\,
\Gamma(\nu+1)\bigl(\frac12\myp\alpha\left(1-\theta\right)^{1/2}\bigr)^{-\nu-1}}{\frac12\mypp\Gamma(\nu)\bigl(\frac12\myp\alpha\left(1-\theta\right)^{1/2}\bigr)^{-\nu}}=\frac{\nu}{\frac12\myp\alpha\left(1-\theta\right)^{1/2}},
\end{equation}
where we also used the recurrence property of the gamma function, $\Gamma(\nu+1)=\nu\,\Gamma(\nu)$
\cite[5.5.1]{NIST}. Substituting this into \eqref{eq:eta} gives \eqref{eq:eta1}. If $\alpha=0$ then \eqref{eq:eta1} readily follows from
\eqref{eq:ex01}.

\item[\rm (b)]
For $\alpha>0$, formulas \eqref{eq:KK} with $\nu=0$ and $\nu=1$ give
\begin{equation}\label{eq:eta1b}
\frac{K_{1}\bigl(\alpha\left(1-\theta\right)^{1/2}\bigr)}{K_{0}\bigl(\alpha\left(1-\theta\right)^{1/2}\bigr)}
\sim \frac{\bigl(\frac12\myp\alpha\left(1-\theta\right)^{1/2}\bigr)^{-1}}{-\log \left(1-\theta\right)},
\end{equation}
and the first formula in \eqref{eq:eta2} follows from \eqref{eq:eta}.
If $\alpha=0$ then formula  \eqref{eq:ex02} immediately gives the second formula in \eqref{eq:eta2}.

\item[\rm (c)]
For $\alpha>0$, using the symmetry relation  \eqref{eq:K1}, similarly to \eqref{eq:eta1a} we obtain
\begin{align*}
\frac{K_{\nu+1}\bigl(\alpha\left(1-\theta\right)^{1/2}\bigr)}{K_{\nu}\bigl(\alpha\left(1-\theta\right)^{1/2}\bigr)}
\sim
\frac{\Gamma(\nu+1)\bigl(\frac12\myp\alpha\left(1-\theta\right)^{1/2}\bigr)^{-2\myp\nu-1}}{\Gamma(-\nu)},
\end{align*}
and the first formula in \eqref{eq:eta3} then follows from \eqref{eq:eta}.
The second formula in \eqref{eq:eta3} is immediate from \eqref{eq:ex03}.

\item[\rm (d)]
Follows from \eqref{eq:eta} using the symmetry relation
\eqref{eq:K1} and the asymptotic ratio \eqref{eq:eta1b}.

\item[\rm (e)] Again using \eqref{eq:K1} and the first line of \eqref{eq:KK} with orders $-\nu>0$ and $-\nu-1>0$, we obtain
\begin{equation*}
\frac{K_{\nu+1}\bigl(\alpha\left(1-\theta\right)^{1/2}\bigr)}{K_{\nu}\bigl(\alpha\left(1-\theta\right)^{1/2}\bigr)}
\sim \frac{\Gamma(-\nu-1)\bigl(\frac12\myp\alpha\left(1-\theta\right)^{1/2}\bigr)}{\Gamma(-\nu)},
\end{equation*}
and \eqref{eq:eta5} follows from \eqref{eq:eta}, again using the recurrence $\Gamma(z+1)=z\,\Gamma(z)$
\cite[5.5.1]{NIST}, now with $z=-\nu-1$.
\end{itemize}
Thus, the proof of Proposition \ref{pr:eta} is complete.
\end{proof}

Proposition \ref{pr:eta} describes the growth of the expected value $\eta$ (for $\alpha>0$) or $\check{\eta}$ (for $\alpha=0$) in terms of the small parameter $1-\theta$. For the purposes of the GIGP model fitting, it is useful to express $1-\theta$ through $\eta$ or $\check{\eta}$, respectively, by solving the asymptotic equations \eqref{eq:eta1}, \eqref{eq:eta2}, \eqref{eq:eta3}, and \eqref{eq:eta4}.

\begin{proposition}\label{pr:theta}
Under the conditions of Proposition \ref{pr:eta}, the following asymptotics hold.
 \begin{itemize}
\item[\rm(a)]
$\nu>0$, \,$\alpha\ge 0$\mypp\textup{:}
\begin{equation*}
1-\theta\sim\frac{\nu}{\eta}.
\end{equation*}
\item[\rm(b)]
$\nu=0$, \,$\alpha\ge0$\mypp\textup{:}
\begin{equation*}
1-\theta\sim \frac{1}{\eta\myp \log\eta},\qquad 1-\theta\sim \frac{1}{\check{\eta}\myp \log\check{\eta}}.
\end{equation*}
\item[\rm(c)]
$-1<\nu<0$, \,$\alpha\ge0$\mypp\textup{:}
\begin{equation*}
1-\theta\sim\left(\frac{\Gamma(\nu+1)\bigl(\tfrac12\myp\alpha\bigr)^{-2\myp\nu}}{\Gamma(-\nu)\,\eta}\right)^{\!1/(\nu+1)}\!,\qquad 1-\theta\sim\left(\frac{-\nu}{\check{\eta}}\right)^{\!1/(\nu+1)}\!.
\end{equation*}

\item[\rm(d)]
$\nu=-1$, \,$\alpha>0$\mypp\textup{:}
\begin{equation*}
\log\myn(1-\theta)\sim -\frac{4\myp\eta}{\alpha^2}.
\end{equation*}
\end{itemize}
\end{proposition}
\begin{remark}
Formula \eqref{eq:eta4} provides only the logarithmic asymptotics of $1-\theta$, but this suffices for the estimation purposes.
\end{remark}

\section{The limit shape in the GIGP model}\label{sec:4}

\subsection{Scaling coefficients and the main theorem}\label{sec:4.1}
Let the frequencies $f_j$ ($j\in\NN_0$) be given by the GIGP distribution formula \eqref{eq:GIGP} with parameters $0<\theta<1$, $\nu\ge -1$ and $\alpha\ge0$, excluding the ``corner'' pair  $\nu=-1$, $\alpha=0$. The case $\alpha=0$ is understood as the limit of conditional probabilities $\PP(X_i=j\,|\mypp X_i>0)=f_j/(1-f_0)$ ($j\in\NN$) as $\alpha\to0+$ (see Section \ref{sec:3.2}).

Given the random vector of observed multiplicities $(M_j)$ produced by $M$ sources, our aim is to study the asymptotics of scaled Young diagrams with the boundary (see \eqref{eq:Y-tilde})
\begin{equation}\label{eq:tildeY}
\widetilde{Y}(x):=\frac{Y(A\myp x)}{B}=\frac{1}{B}\sum_{j\ge
Ax} M_j=\frac{1}{B}\sum_{i=1}^M Z_i(A\myp x)\qquad (x\ge0).
\end{equation}
We proceed under the following assumptions on the limiting regime, including the specification of the scaling coefficients $A$ and $B$.

\begin{assumption}\label{as:theta}
The number of sources is large,  $M\to\infty$. In addition, the intrinsic parameter $\theta\in(0,1)$ is assumed to be close to its upper limit $1$, that is, $\theta\to1-$, which guarantees that the  mean number of items per source is large (see Proposition \ref{pr:eta}).
\end{assumption}

\begin{assumption}\label{as:AB}
The $x$-scaling coefficient $A$ is chosen to be
\begin{equation}\label{eq:A}
A=\frac{1}{-\log \theta}\sim \frac{1}{1-\theta}\to\infty\qquad (\theta\to1-),
\end{equation}
whereas the $y$-scaling coefficient $B$ is specified according to particular domains in the space of parameters $\nu$ and $\alpha$ as follows:
\begin{itemize}
\item[(a)]
$\nu>0$, $\alpha\ge 0$\mypp\textup{:}
\begin{equation}\label{eq:B1}
B=\frac{M}{\Gamma(\nu)}.
\end{equation}
\item[(b)]
$\nu=0$, $\alpha\ge0$\mypp\textup{:}
\begin{equation}\label{eq:B2}
 B=\frac{M}{-\log\myn(1-\theta)}.
\end{equation}
\item[(c)]
$-1\le \nu<0$, $\alpha>0$\mypp\textup{:}
\begin{equation}\label{eq:B3}
B=\frac{\displaystyle M\mypp\bigl(\tfrac12\myp\alpha\bigr)^{-2\myp\nu}
\mynn \left(1-\theta\right)^{-\nu}
}{\Gamma(-\nu)}.
\end{equation}
\item[(d)]
$-1<\nu<0$, $\alpha=0$\mypp\textup{:}
\begin{equation}\label{eq:B4}
B=\frac{M\mypp(-\nu)\myn \left(1-\theta\right)^{-\nu}}{\Gamma(\nu+1)}.
\end{equation}
\end{itemize}
\end{assumption}

\begin{assumption}\label{as:B}
The $y$-scaling coefficient $B$ defined in Assumption \ref{as:AB} is large, $B\to\infty$. For $\nu>0$, this is automatic according to \eqref{eq:B1} (as long as $M\to\infty$), but for $\nu\le0$ we must assume in addition that $M\gg-\log\myn(1-\theta)$ if $\nu=0$ and $M\gg \left(1-\theta\right)^{\nu}$ if $\nu<0$.
\end{assumption}
\begin{remark}\label{rm:M}
The need to impose an additional condition in Assumption \ref{as:B} on the joint limiting behavior of the external parameter $M\to\infty $ and the intrinsic GIGP parameter $\theta\to1-$ for $\nu\le0$ shows that, in order to have a manifested limit shape in the data, the number of sources, $M$, must be sufficiently large. We will clarify the opposite situation below in Section  \ref{sec:6}.
\end{remark}

For $\nu\ge -1$, consider the function
\begin{equation}\label{eq:phi_nu} \varphi_\nu(x):=\int_x^\infty\! s^{\nu-1}\mypp \rme^{-s}\,\rmd{s}\qquad (x>0),
\end{equation}
which is the \emph{(upper) incomplete gamma function} \cite[8.2.2]{NIST}. The following is our main result, establishing convergence in probability of the scaled Young diagrams (see \eqref{eq:tildeY}) to the limit shape $\varphi_\nu(x)$.

\begin{theorem}\label{th:main}
Under Assumptions \ref{as:theta}, \ref{as:AB} and \ref{as:B}, for any $\varepsilon >0$ and any $\delta>0$ we have
\begin{equation}\label{eq:LS0}
\PP\left(\sup_{x\ge\delta}\mypp\bigl|\widetilde{Y}(x)-\varphi_\nu(x)\bigr|\ge \varepsilon\right)\to0.
\end{equation}
\end{theorem}

The proof of Theorem \ref{th:main} is developed below in Sections \ref{sec:4.2} to \ref{sec:4.6}.

\subsection{Graphical illustration using computer simulations}\label{sec:4.11}
In this section, we illustrate the limit shape approximation using computer simulated data in two example cases, with $\nu=0.5$ and $\nu=-0.5$ (see Fig.\,\ref{GIGPsample}, left panels). The other parameter settings are as follows, $\alpha=2$. $\theta=0.99$, and $M=1\myp 000$. The plots depict the data as the upper boundary of the Young diagram $Y(x)$ defined in \eqref{eq:Y} and the theoretical GIGP complementary distribution function $\bar{F}(x)$ (see \eqref{eq:Z1} and \eqref{eq:GIGP}), along with the limit shape scaled back to the original frequencies of counts, that is, $x\mapsto B\,\varphi_\nu(x/A)$, where $A=-1/\log\theta\doteq 99.49916$ (see \eqref{eq:A}) and $B\doteq 564.1896$ for $\nu=0.5$ or $B\doteq 56.41896$ for $\nu=-0.5$ (see \eqref{eq:B1} and \eqref{eq:B3}, respectively). In both cases, the plots show a very good fit of the limit shape in the bulk of the observed values.

\begin{figure}[h!]
\renewcommand{\thesubfigure}{}
\centering
\subfigure[\raisebox{4.2pc}{(a) $\nu=0.5$}]
{\includegraphics[width=.47\textwidth]{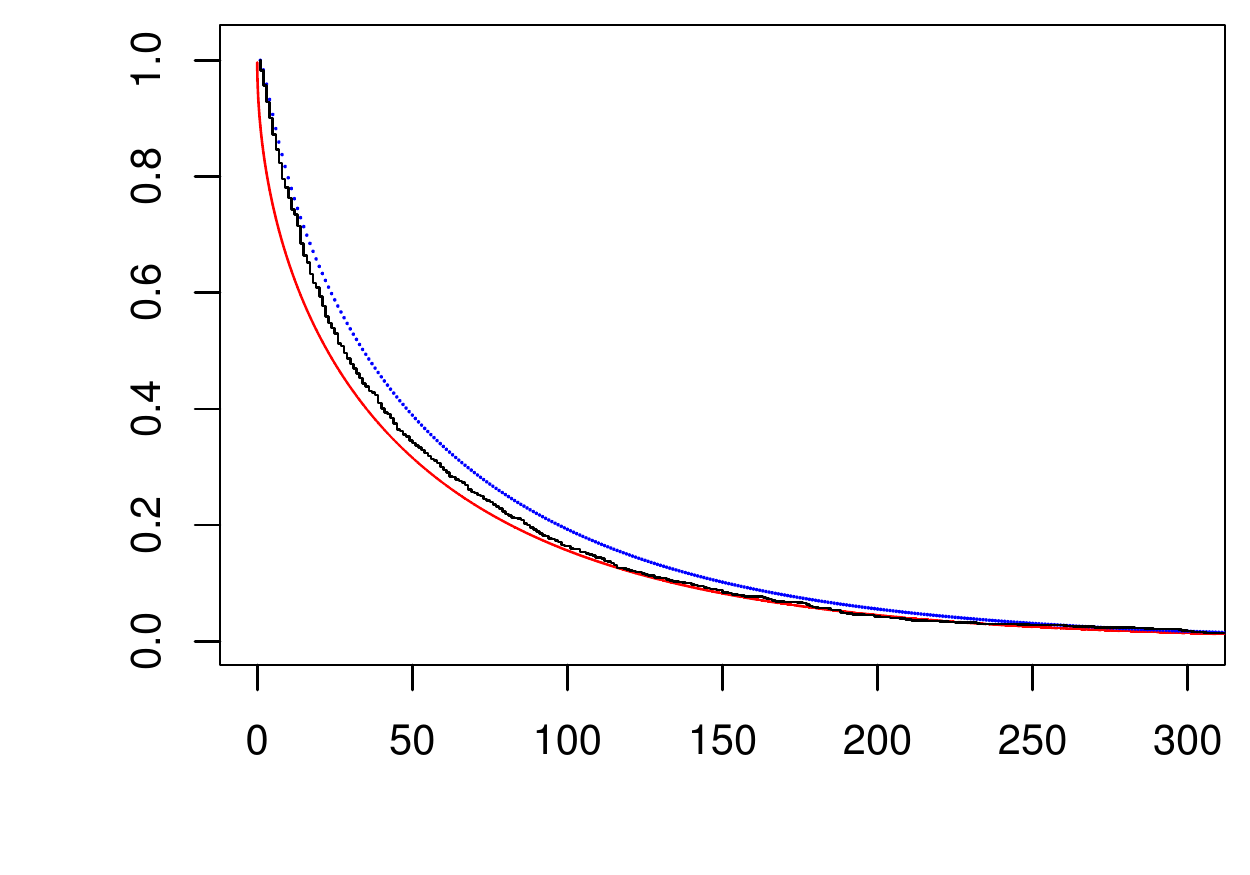}\vspace{-2.3pc}
\put(-90,6){\mbox{\footnotesize$x$}}
\put(-198,85){\mbox{\footnotesize$y$}}
\mbox{}\hspace{1.2pc}
\includegraphics[width=.47\textwidth]{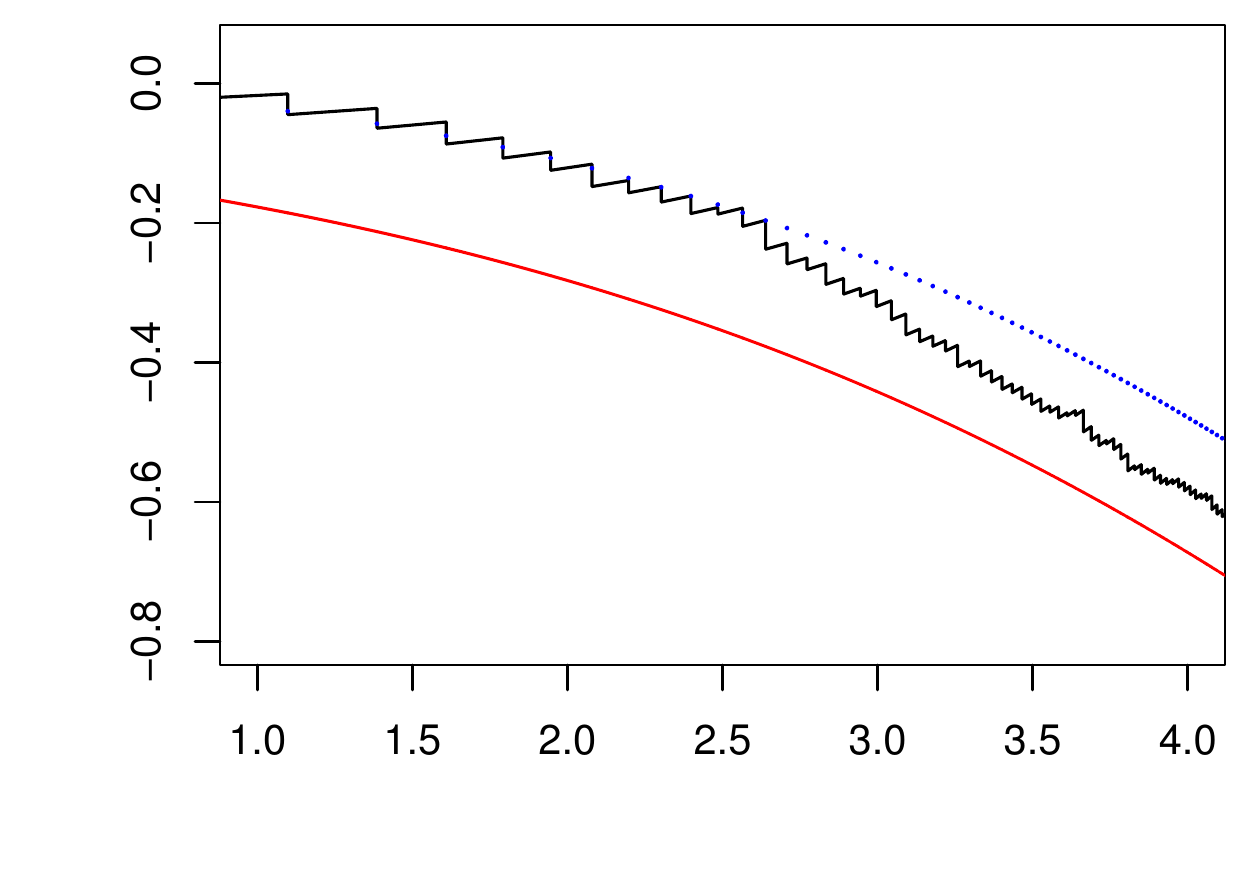}
\put(-90,6){\mbox{\footnotesize$u$}}
\put(-198,85){\mbox{\footnotesize$v$}}
}
\\\vspace{-3.1pc}
\subfigure[\raisebox{4.2pc}{(b) $\nu=-0.5$}] {\includegraphics[width=.47\textwidth]{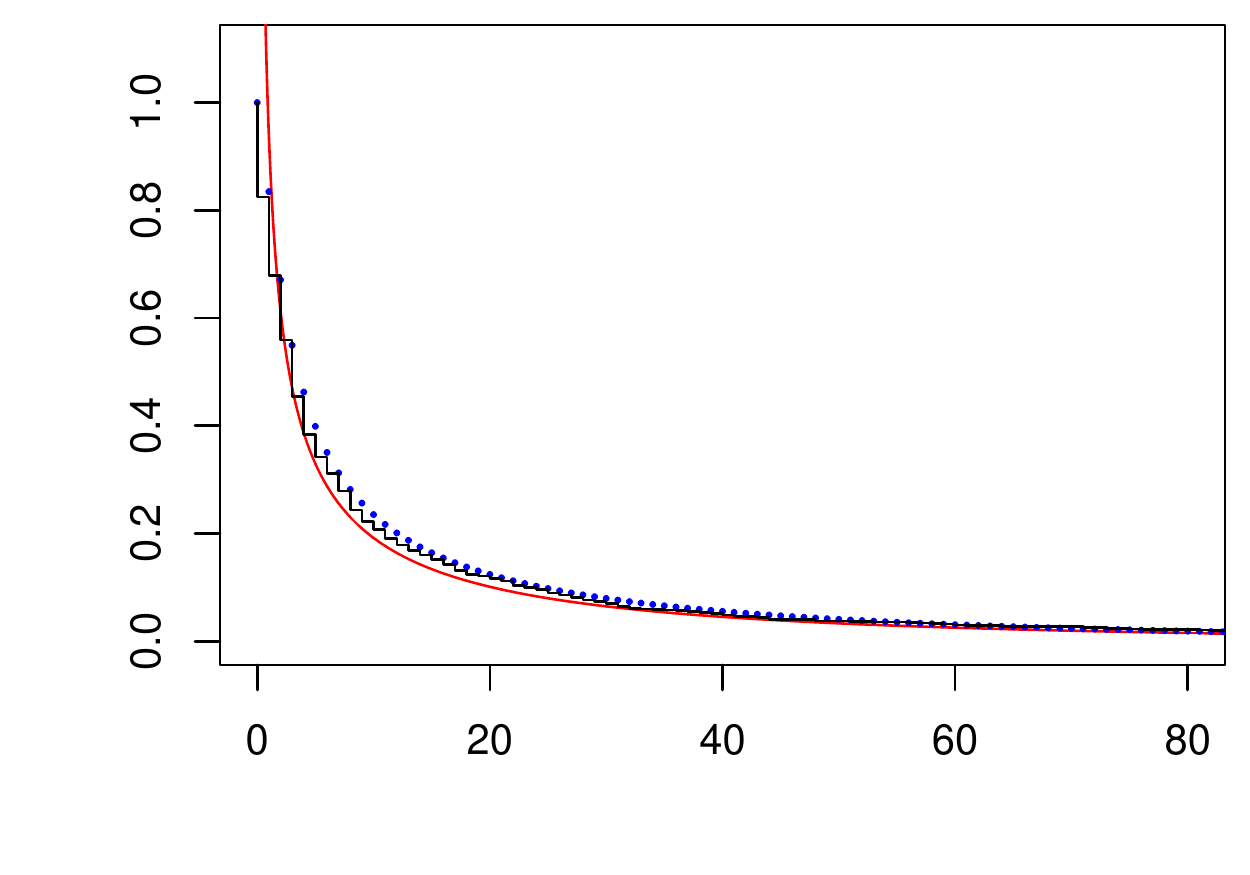}
\put(-90,6){\mbox{\footnotesize$x$}}
\put(-198,83){\mbox{\footnotesize$y$}}
\mbox{}\hspace{1.2pc}
\includegraphics[width=.47\textwidth]{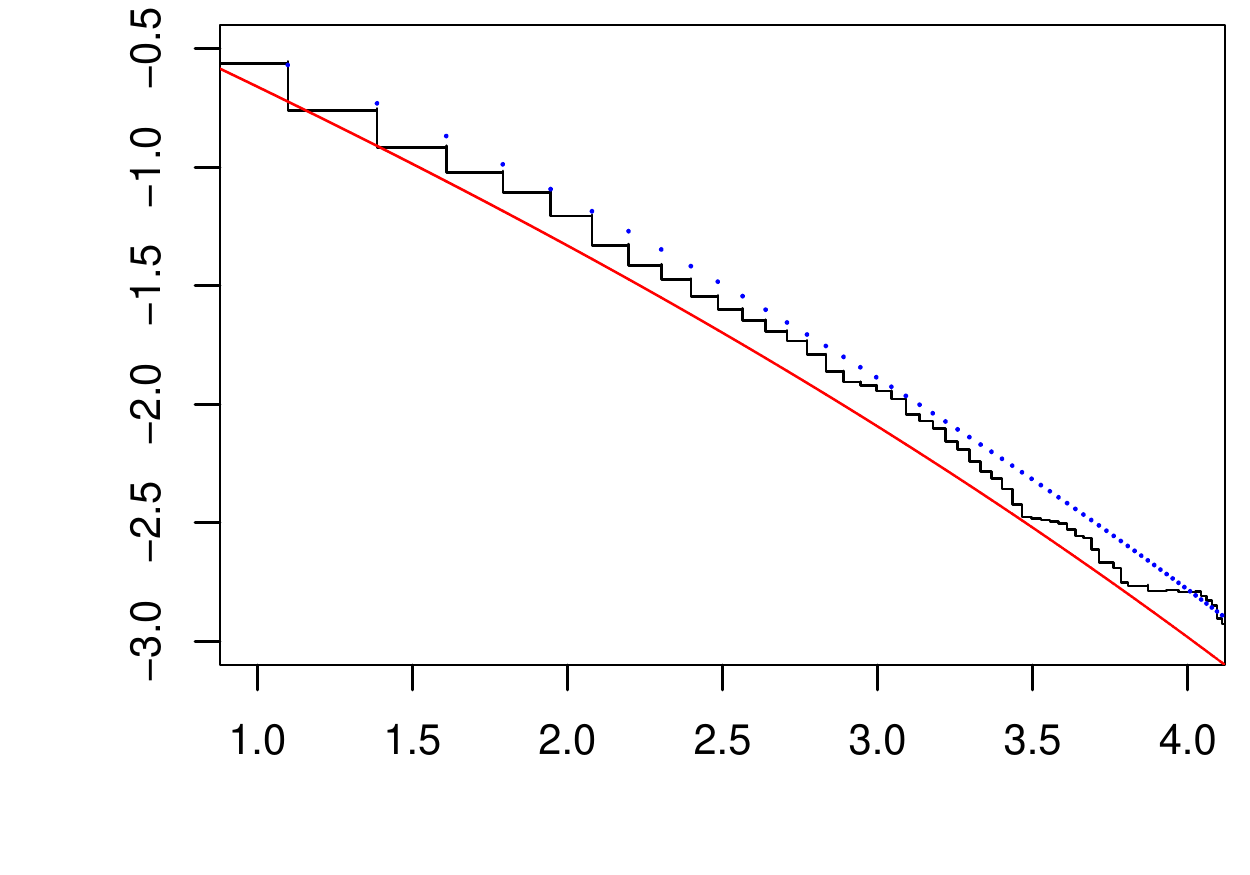}
\put(-90,6){\mbox{\footnotesize$u$}}
\put(-198,83){\mbox{\footnotesize$v$}}
}\vspace{-3.8pc}
\caption{Illustration of the limit shape approximation using $M=1{}\myp000$ random values $(X_i)$ simulated using the GIGP model \eqref{eq:GIGP} with parameters $\theta=0.99$, $\alpha=2$, and (a) $\nu=0.5$ or (b) $\nu=-0.5$.
In the left panels, the black stepwise plots represent the upper boundary $Y(x)$ of the corresponding Young diagrams, together with the GIGP complementary distribution function $\bar{F}(x)$ shown as blue dotted plots, while the smooth red curves represent the scaled back limit shape, $x\mapsto B\,\varphi_\nu(x/A)$. In the right panels, the tails are shown in transformed coordinates \eqref{eq:uv}, with the same line and color coding.}
    \label{GIGPsample}
\end{figure}
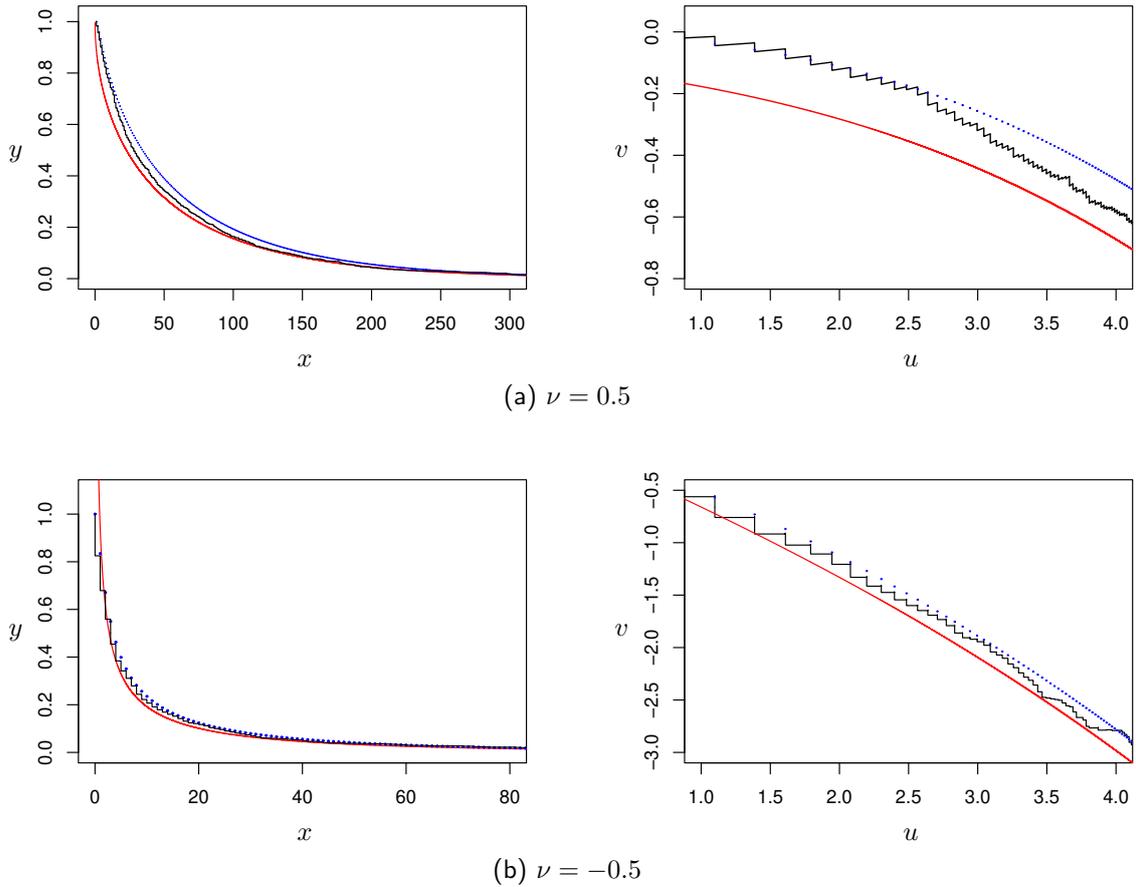
The inspection of the tail behavior is facilitated by observing from \eqref{eq:phi_nu}  that
\begin{align*}
\varphi_\nu(x&)=-\int_x^\infty\! s^{\nu-1}\mypp \rmd\myp{(\rme^{-s})}\\&=x^{\nu-1}\,\rme^{-x}+(\nu-1)\!\int_x^\infty\! s^{\nu-2}\, \rme^{-s}\,\rmd{s}\sim x^{\nu-1}\,\rme^{-x}\qquad (x\to\infty).
\end{align*}
Therefore, according to \eqref{eq:tildeY} and \eqref{eq:LS0},
it may be expected that, for large enough $x$,
$$
y=Y(A\myp x)\approx B\,\varphi_\nu(x)\approx B\,x^{\nu-1}\mypp \rme^{-x},
$$
or, taking the logarithm,
\begin{equation}\label{eq:logB}
\log Y(A\myp x)+x\approx \log B +(\nu-1)\log x.
\end{equation}
Hence, switching from $(x, y)$ to the new coordinates
\begin{equation}\label{eq:uv}
u=\log x,\qquad v=\log y+x,
\end{equation}
a transformed data plot may be expected to be close to a straight line with slope $\nu-1$, as well as the tails of the theoretical GIGP distribution function and of the limit shape alike. This is illustrated for the simulated data in Fig.\,\ref{GIGPsample} (right panels), showing a reasonable linearization of the long tails in both cases, $\nu=0.5$ and $\nu=-0.5$.

The graphical method described above can be used for a quick visual check of suitability of the GIGP frequency model even before estimating the model parameters, by first experimenting with the scaling coefficient $A=-1/\log \theta$ (see \eqref{eq:A}) aiming to get a linearized data plot (thus producing a crude estimate for the parameter $\theta$), followed by reading off the fitted slope (which estimates the parameter $\nu-1$), and then exploiting the fitted intercept (close to $\log B$, see \eqref{eq:logB}) to get an estimate for the parameter $\alpha$ using one of the formulas \eqref{eq:B1} to \eqref{eq:B4}. We will apply this method to some real data sets in Section \ref{sec:7}.

\subsection{Convergence of expected Young diagrams}\label{sec:4.2}

We start our proof of Theorem \ref{th:main} by showing that convergence to the limit shape $\varphi_\nu(x)$ holds for the expected Young diagrams. From \eqref{eq:Y-exp_var} and \eqref{eq:Y-tilde}, we have
\begin{equation}
\label{eq:expY}
\EE\bigl(\widetilde{Y}(x)\bigr) = \frac{M\myp\bar{F}(A\myp x)}{B}.
\end{equation}

\begin{theorem}\label{th:conv_exp}
Under Assumptions \ref{as:theta} and \ref{as:AB},
\begin{equation}\label{eq:E-phi}
\EE\bigl(\widetilde{Y}(x)\bigr)\to \varphi_\nu(x) \qquad(x>0),
\end{equation}
uniformly in $x\ge \delta$ for any $\delta>0$.
\end{theorem}

\begin{remark}
Note that Assumption \ref{as:B} is not needed in Theorem \ref{th:conv_exp}.
    \end{remark}
The following useful criterion for  uniform convergence of monotone functions (adapted to the half-line domain) is well known (see, e.g., \cite[Sec.\,0.1]{Resnick}).
\begin{lemma}\label{lm:uniform}
Let a sequence of monotone functions on $(0,\infty)$, uniformly bounded
on $[\myp\delta,\infty)$ for any $\delta>0$, converge pointwise to a continuous (monotone) function. Then this convergence is uniform on $[\myp\delta,\infty)$, for any $\delta>0$.
\end{lemma}

Noting that the limiting function $\varphi_\nu(x)$ in \eqref{eq:phi_nu} is continuous and monotone decreasing on $(0,\infty)$, by Lemma \ref{lm:uniform} it suffices to prove pointwise convergence \eqref{eq:E-phi}, for each $x>0$.

\begin{remark}
In calculations below, we confine ourselves to the leading asymptotics \eqref{eq:Kj1} of terms in the series $\bar{F}(A\myp x)$ (see \eqref{eq:expY}). A more careful analysis involving control over the approximation errors is straightforward by using the classic Euler--Maclaurin summation formula
\cite[\S\mypp2.10(i)]{NIST}
and uniform asymptotic expansions of the Bessel function of large order \cite[\S\myp10.41(ii)]{NIST}.
\end{remark}
\begin{proof}[Proof of Theorem \ref{th:conv_exp}]
The proof below is broken down according to various sub-domains of the parameters $\nu$ and $\alpha$ (see Assumption~\ref{as:AB}). First, we consider the cases with $\alpha>0$, where the GIGP distribution is supported on $j\in\NN_0$, and then switch to the boundary cases with $\alpha=0$, where the support is reduced to $j\in\NN$.

\begin{itemize}
\item $\alpha>0$

Using the asymptotic approximation \eqref{eq:Kj1} of the frequencies $f_j$ (with $j\ge A\myp x\ge A\myp\delta\gg1$), from \eqref{eq:expY} we obtain
\begin{equation}
 \label{eq:LS2}
\EE\bigl(\widetilde{Y}(x)\bigr)
=\frac{M}{B}\sum_{j\ge Ax} f_j\sim \frac{M\left(1-\theta\right)^{\nu/2}\bigl(\tfrac12\myp\alpha\bigr)^{-\nu}}{2 B\myp
K_\nu\bigl(\alpha\left(1-\theta\right)^{1/2}\bigr)}\sum_{j\ge Ax} j^{\nu-1} \myp\theta^j. \end{equation}
Recalling that $A=\left(-\log\theta\right)^{-1}\sim\left(1-\theta\right)^{-1}$, for the last sum in \eqref{eq:LS2} we have
\begin{equation}
\label{eq:LS2a}
A^{-\nu} \!\sum_{j\ge Ax} j^{\nu-1} \myp\theta^j=\sum_{j\ge Ax}\left(\frac{j}{A}\right)^{\nu-1}\mynn \rme^{-j/A}\,\frac{1}{A}\to \int_x^\infty \! s^{\nu-1}\mypp\rme^{-s}\,\rmd{s}=\varphi_\nu(x),
\end{equation}
which is evident by interpreting \eqref{eq:LS2a} as the Riemann integral sum converging to the integral on the right.
Furthermore, the asymptotics of the denominator in \eqref{eq:LS2} is obtained from formulas \eqref{eq:KK} (see Lemma \ref{lm:K}).
Hence, returning to
\eqref{eq:LS2} and recalling the definitions \eqref{eq:A} of $A$ and \eqref{eq:B1}, \eqref{eq:B2}, \eqref{eq:B3} of $B$, we easily obtain \eqref{eq:E-phi}.

\item $\alpha=0$

Using the tail approximations \eqref{eq:Kj2} ($\nu>0$), \eqref{eq:f02} ($\nu=0$) and \eqref{eq:Kj3} ($-1<\nu<0$), from \eqref{eq:expY} we obtain, similarly to \eqref{eq:LS2} and \eqref{eq:LS2a},
\begin{equation}
 \label{eq:LS3}
\EE\bigl(\widetilde{Y}(x)\bigr)
\sim \frac{M\myp C_\nu(\theta)}{B}
\sum_{j\ge Ax} j^{\nu-1} \myp\theta^j\sim  \frac{M\myp C_\nu(\theta)\myp A^\nu}{B}\,\varphi_\nu(x), \end{equation}
where $A\sim \left(1-\theta\right)^{-1}$ (see \eqref{eq:A}) and
$$
C_\nu(\theta):=\begin{cases}
 \displaystyle(1-\theta)^\nu\mynn/\myp\Gamma(\nu)&(\nu>0),\\[.3pc]
(-\log\myn(1-\theta))^{-1}&(\nu=0),\\[.3pc]
\displaystyle (-\nu)/\myp\Gamma(\nu+1)&(-1<\nu<0).\end{cases}
$$
Now, using the specifications \eqref{eq:B1}, \eqref{eq:B2}, or \eqref{eq:B4}, it is immediate to see that the right-hand side of \eqref{eq:LS3} is reduced to $\varphi_\nu(x)$.
\end{itemize}
This completes the proof of Theorem \ref{th:conv_exp}.
\end{proof}

\subsection{Pointwise convergence of random Young diagrams}\label{sec:4.3}

Before addressing a stronger Theorem \ref{th:main} stating the uniform convergence in probability,
we start with a simpler statement about pointwise convergence of $\widetilde{Y}(x)$ for any $x>0$.

\begin{theorem}\label{th:conv_Y}
Under Assumptions \ref{as:theta}, \ref{as:AB} and \ref{as:B}, the mean squared deviation of\/ $\widetilde{Y}(x)$ from the limit shape $\varphi_\nu(x)$  is asymptotically small,
\begin{equation}\label{eq:Esq}
\EE\bigl(\bigl|\widetilde{Y}(x)-\varphi_\nu(x)\bigr|^2\bigr)\to0,
\end{equation}
 uniformly in $x\ge \delta$ for any $\delta>0$. This implies convergence in probability, $\widetilde{Y}(x)\stackrel{{\rm p}\myp}{\to}\varphi_\nu(x)$, that is, for any $\varepsilon>0$,
\begin{equation}\label{eq:Psq}
\sup_{x\ge\delta}\mypp\PP\bigl(\bigl|\widetilde{Y}(x)-\varphi_\nu(x)\bigr|\ge \varepsilon)\to0.
\end{equation}
\end{theorem}
\begin{proof}
By the standard decomposition of the mean squared deviation, we have
\begin{equation}\label{eq:Var+Esq}
\EE\bigl(\bigl|\widetilde{Y}(x)-\varphi_\nu(x)\bigr|^2\bigr)=\Var\bigl(\widetilde{Y}(x)\bigr)+\bigl(\EE\bigl(\widetilde{Y}(x)\bigr)-\varphi_\nu(x)\bigr)^2.
\end{equation}
Using formulas \eqref{eq:Y-exp_var} and \eqref{eq:Y-tilde}, the variance term in
\eqref{eq:Var+Esq} is estimated as follows,
\begin{align}
\notag
\Var\bigl(\widetilde{Y}(x)\bigr)&=\frac{M\bar{F}(A\myp x)\myp F(A\myp x)}{B^2} \\
&\le \frac{M\bar{F}(A\myp x)}{B^2}\sim \frac{\varphi_\nu(x)}{B}\to0,
\label{eq:to0}
\end{align}
according to \eqref{eq:expY}, \eqref{eq:E-phi}, and also Assumption \ref{as:B}, which guarantees that $B\to\infty$. By Theorem \ref{th:conv_exp}, convergence in \eqref{eq:to0} is uniform on $[\myp\delta,\infty)$, for every $\delta>0$. As for the second term on the right-hand side of \eqref{eq:Var+Esq}, due to Theorem \ref{th:conv_exp} it is asymptotically small, uniformly on every interval $[\myp\delta,\infty)$. Hence,
the limit \eqref{eq:Esq} follows.

Finally, convergence in probability \eqref{eq:Psq} is a standard consequence of \eqref{eq:Esq} due to Chebyshev's inequality \cite[Sec.\,II.6, p.\mypp192]{Shiryaev}.
\end{proof}

\subsection{Auxiliary martingale}\label{sec:4.4}

Recalling the definition \eqref{eq:Zi} and interpreting $Z_i(x)$ as a random process with ``time'' $x\in [\myp0,\infty)$, consider a rescaled process in inverted time $t=1/x$,
\begin{equation}\label{eq:tilde=Z}
\widetilde{Z}_i(t):=\frac{1-Z_i(1/t)}{F(1/t)}=\frac{I_{\{X_i<1/t\}}}{F(1/t)}
=\begin{cases}
1/F(1/t),&0<t<1/X_i,\\
0,& 1/X_i\le t<\infty.
\end{cases}
\end{equation}
Note that, with probability $1$,
$$
\lim_{t\to0+}\widetilde{Z}_i(t)=\frac{I_{\{X_i<\infty\}}}{F(\infty)}=1,
$$
so $\widetilde{Z}_i(t)$ can be extended by continuity to the origin by setting $\widetilde{Z}_i(0):=1$. Clearly,
\begin{align}
\label{eq:ExpW}
\EE\bigl(\widetilde{Z}_i(t)\bigr)&=\frac{\PP(X_i<1/t)}{F(1/t)}=1,\\ \Var\bigl(\widetilde{Z}_i(t)\bigr)&=\frac{F(1/t)\bigl(1-F(1/t)\bigr)}{F(1/t)^2}=\frac{\bar{F}(1/t)}{F(1/t)},
\label{eq:VarW}
\end{align}
using that $\bar{F}(1/t)=1-F(1/t)$.

Let $\mathcal{F}_i(t)=\sigma\{\widetilde{Z}_i(s),\mypp s\le t\}$ ($t\ge 0$) denote the smallest sigma-algebra containing all events $\{Z_i(1/s)=1\}=\{X_i\ge 1/s\}$ with $0\le s\le t$. Consider the product sigma-algebra $\mathcal{F}(t)=\mathcal{F}_1(t)\otimes\dots\otimes \mathcal{F}_M(t)$, and define the random process
\begin{equation}\label{eq:W}
W(t):=\widetilde{Z}_1(t)+\dots+\widetilde{Z}_M(t) -M\qquad (t\ge 0).
\end{equation}
From \eqref{eq:tilde=Z}, it is easy to see that the process $W$ is \emph{c\`adl\`ag} (i.e., its paths are everywhere right-continuous
and have left limits). Furthermore, we have (see \eqref{eq:ExpW} and \eqref{eq:VarW})
\begin{equation}\label{eq:ExpVarW}
\EE\bigl(W(t)\bigr)=0,\qquad \Var\bigl(W(t)\bigr)=\frac{M\myp\bar{F}(1/t)}{F(1/t)},
\end{equation}
since the random variables $\widetilde{Z}_1(t), \dots, \widetilde{Z}_M(t)$ in the sum \eqref{eq:W} are mutually independent.

\begin{lemma}\label{lm:W}
The process $W(t)$ is a martingale with respect to the filtration $\mathcal{F}(t)$  $(t\ge 0)$, that is, for any $t\ge s\ge0$ we have $\EE\bigl(W(t)\mypp|\mypp\mathcal{F}(s)\bigr)=W(s)$, with probability $1$,
\end{lemma}

\begin{proof}
Let $t\ge s\ge0$. Due to definition \eqref{eq:W} and independence of $(X_i)$, we have
\begin{equation}\label{eq:EW}
\EE\bigl(W(t)\mypp|\mypp\mathcal{F}(s)\bigr)=\sum_{i=1}^M \EE\bigl(\widetilde{Z}_i(t)\mypp|\mypp\mathcal{F}_i(s)\bigr)-M,
\end{equation}
so it suffices to prove that
\begin{equation}\label{eq:EZi}
\EE\bigl(\widetilde{Z}_i(t)\mypp|\mypp\mathcal{F}_i(s)\bigr)=\widetilde{Z}_i(s).
\end{equation}

By definition, information contained in the sigma-algebra $\mathcal{F}_i(s)$ relates to the threshold events of the form $\{X_i\ge 1/r\}$ for all $0\le r\le s$. In view of the target random variable in \eqref{eq:EZi} expressed through $Z_i(1/t)=I_{\{X_i\ge 1/t\}}$, this essentially amounts to the knowledge of whether $X_i\ge 1/s$ or not.

Suppose first that $X_i<1/s$, that is, $\widetilde{Z}_i(s)=1/F(1/s)$ (see \eqref{eq:tilde=Z}).
Then
\begin{align*}
\EE\bigl(\widetilde{Z}_i(t)\mypp|\mypp X_i<1/s\bigr)&=\frac{\PP(X_i<1/t\,|\mypp X_i<1/s)}{F(1/t)}\\
&=\frac{\PP(X_i<1/t)}{F(1/t)\,\PP(X_i<1/s)}\\
&=\frac{1}{F(1/s)}=\widetilde{Z}_i(s).
\end{align*}

Similarly, if $X_i\ge 1/s\ge 1/t$ then  $\widetilde{Z}_i(s)=0$ and
\begin{align*}
\EE\bigl(\widetilde{Z}_i(t)\mypp|\mypp X_i\ge 1/s\bigr)&=\frac{\PP(X_i<1/t\,|\mypp X_i\ge 1/s)}{F(1/t)}\\
&=\frac{\PP(X_i<1/t,  X_i\ge 1/s)}{F(1/t)\,\PP(X_i\ge 1/s)}\\
&=0=\widetilde{Z}_i(s).
\end{align*}

Thus, the relation \eqref{eq:EZi} holds true, and in view of \eqref{eq:EW} and the definition \eqref{eq:W} the proof of Lemma \ref{lm:W} is complete.
\end{proof}

The martingale $W(t)$ will be used in the next section for the estimation of the uniform distance between the functions $\widetilde{Y}(x)$ and $\varphi_\nu(x)$.

\subsection{Uniform convergence of random Young diagrams}\label{sec:4.5}

Pointwise convergence established in  Theorem \ref{th:conv_Y} can be strengthened to the uniform convergence away from the origin, yielding our main result stated above as Theorem \ref{th:main}.

\begin{proof}[Proof of Theorem \ref{th:main}]
Note that
$$
\sup_{x\ge\delta}\myn\bigl|\widetilde{Y}(x)-\varphi_\nu(x)\bigr|\le \sup_{x\ge\delta}\myn\bigl|\widetilde{Y}(x)-\EE\bigl(\widetilde{Y}(x)\bigr)\bigr|+
\sup_{x\ge\delta}\myn\bigl|\EE\bigl(\widetilde{Y}(x)-\varphi_\mu(x)\bigr)\bigr|,
$$
where $\EE\bigl(\widetilde{Y}(x)\bigr)=M \bar{F}(A\myp x)/B$ (see \eqref{eq:expY}).  Hence, by virtue of Theorem \ref {th:conv_exp}, it suffices to consider the deviations
$$
\sup_{x\ge\delta}\mynn\left|\widetilde{Y}(x)-\frac{M \bar{F}(A\myp x)}{B}\right|\ge \varepsilon.
$$

Using the definitions \eqref{eq:tilde=Z}, \eqref{eq:W} and the relation $\bar{F}(x)=1-F(x)$, we have
\begin{align*}
\widetilde{Y}(x)-\frac{M \bar{F}(A\myp x)}{B}&=\frac{1}{B}\left(\sum_{i=1}^M Z_i(A\myp x)-M\bar{F}(A\myp x)\right)\\
&=-\frac{1}{B}\left(\sum_{i=1}^M \bigl(1-Z_i(A\myp x)\bigr) -M F(A\myp x)\right)\\
&=-\frac{F(A\myp x)}{B}\left(\sum_{i=1}^M \widetilde{Z}_i(1/A\myp x)-M\right)\\
&=-\frac{F(A\myp x)}{B}\,W(1/A\myp x).
\end{align*}
Since $0\le F(A\myp x)\le1$ and $t:=1/A\myp x\in[\myp0,1/A\myp\delta\myp]$, this implies
$$
\sup_{x\ge \delta}
\myn\left|\widetilde{Y}(x)-\frac{M \bar{F}(A\myp x)}{B}\right|
\le \frac{1}{B}\sup_{t\le 1/A\myp\delta}\bigl|W(t)\bigr|.
$$
Hence, by the Doob--Kolmogorov submartingale inequality (see, e.g.,
\cite[Theorem 6.16, p.\,101]{Yeh})
applied to the martingale $W(t)$ (see Lemma \ref{lm:W}) and using formulas \eqref{eq:ExpVarW}, we obtain
\begin{align}
\notag
\PP\!\left(\sup_{x\ge \delta}
\myn\left|\widetilde{Y}(x)-\frac{M \bar{F}(A\myp x)}{B}\right|\ge \varepsilon\right)
&\le\PP\!\left( \sup_{t\le 1/A\myp\delta}\bigl|W(t)\bigr|\ge B\myp\varepsilon\right)\\
\notag
&\le \frac{\Var\myp\bigl(W(1/A\myp\delta)\bigr)}{B^2\myp\varepsilon^2
}\\[.2pc]
&=\frac{M\myp\bar{F}(A\myp\delta)}{F(A\myp\delta)\mypp B^2\myp\varepsilon^2}\sim \frac{\varphi_\nu(\delta)}{B\mypp\varepsilon^2}\to0,
\label{eq:Var+}
\end{align}
where at the last step we used that $F(A\myp\delta)\to1$ (as $A\to\infty$) and the limit \eqref{eq:to0} (with $x=\delta$). This completes the proof of Theorem \ref{th:main}.
\end{proof}

\subsection{Young boundary as an empirical process}\label{sec:4.6}

The random function $Y(x)$ given by  \eqref{eq:YZ} can be viewed as an \emph{empirical process} \cite[Ch.\,11]{BLM}, determined by an independent random sample $(X_1,\dots,X_M)$  through a family of test functions $(g_x,\,x\ge 0)$, defined by
\begin{equation}\label{eq:gx}
g_x(j):=\mathbf{1}_{[x,\infty)}(j)\qquad (j\in\NN_0).
\end{equation}
Namely, using \eqref{eq:Zi} and \eqref{eq:YZ}, we can write
$$
Y(x)=\sum_{i=1}^M  Z_i(x)=\sum_{i=1}^M \mathbf{1}_{[x,\infty)}(X_i)=\sum_{i=1}^M g_x(X_i).
$$
The advantage of this approach is that there are sharp upper bounds for the variance of the supremum of an empirical process.

Note that
\begin{align}
\notag
\sup_{x\ge \delta}
\myn\left|\widetilde{Y}(x)-\frac{M \bar{F}(A\myp x)}{B}\right| &\le\frac{1}{B}\sup_{x\ge\delta} \sum_{i=1}^M \myn\bigl|g_{Ax}(X_i)-\bar{F}(A\myp x)\bigr|\\
&=\frac{1}{B}\sup_{x\ge\delta} \sum_{i=1}^M \myn\bigl|\tilde{g}_{x}(X_i)\bigr|,
\label{eq:supY}
\end{align}
where
\begin{equation}\label{eq:g-tilde}
\tilde{g}_x(j):=g_{Ax}(j)-\bar{F}(A\myp x)\qquad (j\in\NN_0).
\end{equation}
According to \cite[Theorem 11.1, pp.\,314, 316]{BLM},
\begin{equation}\label{eq:Var<}
\Var\!\left(\sup_{x\ge\delta} \sum_{i=1}^M\mynn \bigl|\tilde{g}_{x}(X_i)\bigr|\right)\le  \sum_{i=1}^M \EE\!\left(\sup_{x\ge\delta}\myn \bigl|\tilde{g}_{x}(X_i)\bigr|^2\right)\!.
\end{equation}
Using \eqref{eq:gx} and \eqref{eq:g-tilde}, we have
\begin{align*}
 \bigl|\tilde{g}_{x}(X_i)\bigr|^2&\le \left(Z_i(A\myp x)+\bar{F}(A\myp x)\right)^2\\[.2pc]
 &=Z_i(A\myp x)^2+2\mypp Z_i(A\myp x)\bar{F}(A\myp x)+\bar{F}(A\myp x)^2\\[.2pc]
 &\le Z_i(A\myp x)+2\mypp Z_i(A\myp x)+\bar{F}(A\myp x),
\end{align*}
hence
$$
\sup_{x\ge\delta}\myn  \bigl|g_{Ax}(X_i)\bigr|^2=3\mypp Z_i(A\myp\delta)+\bar{F}(A\myp\delta)
$$
and (see  \eqref{eq:Var<})
$$
\EE\!\left(\sup_{x\ge\delta}\myn  \bigl|g_{Ax}(X_i)\bigr|^2\right)\le 3\,\EE\bigl(Z_i(A\myp\delta)\bigr)+\bar{F}(A\myp\delta)=4\mypp \bar{F}(A\myp\delta).
$$

Returning to \eqref{eq:supY}, this gives, together with
Chebyshev's inequality,
\begin{align*}
\PP\!\left(\sup_{x\ge\delta}
\myn\left|\widetilde{Y}(x)-\frac{M\bar{F}(A\myp x)}{B}\right|\ge \varepsilon\right)
&\le\PP\!\left(\sup_{x\ge\delta} \sum_{i=1}^M \bigl|\tilde{g}_{x}(X_i)\bigr|\ge B\myp\varepsilon\right)\\
&\le \frac{4\myp M\bar{F}(A\myp\delta)}{B^2\myp\varepsilon^2}\sim \frac{4\mypp \varphi_\nu(\delta)}{B\mypp\varepsilon^2}\to0,
\end{align*}
using that $B\to\infty$.
Thus, this furnishes an alternative proof of Theorem \ref{th:main} (cf.\ \eqref{eq:Var+}).

\section{Fluctuations of random Young diagrams}\label{sec:5}

Recalling that $\widetilde{Y}(x)$ is a (normalized) sum of independent indicators $Z_i(A\myp x)=I_{\{X_i\ge Ax\}}$, $i=1,\dots,M$ (see \eqref{eq:tildeY}), it is natural to expect that $\widetilde{Y}(x)$ is asymptotically normal, with mean $\EE\bigl(\widetilde{Y}(x)\bigr)=M\bar{F}(A\myp x)/B\sim \varphi_\nu(x)$ and variance $M\bar{F}(A\myp x)\myp F(A\myp x)/B^2\sim \varphi_\nu(x)/B$ (see \eqref{eq:expY} and \eqref{eq:E-phi}).  However, a standard central limit theorem is not directly applicable because the ``success'' probability $\PP(Z_i(A\myp x)=1)=\bar{F}(A\myp x)$ is not constant (and, moreover, it tends to~$0$), so we have to re-prove this statement using the method of characteristic functions.

\begin{theorem}\label{th:CLT1}
Under Assumptions \ref{as:theta}, \ref{as:AB} and \ref{as:B}, for any $x>0$,
\begin{equation}\label{eq:CLT1}
\varUpsilon(x) :=\sqrt{\frac{B}{\varphi_\nu(x)}}\left(\widetilde{Y}(x)-\frac{M\bar{F}(A\myp x)}{B}\right) \stackrel{{\rm d}}{\longrightarrow}\mathcal{N}(0,1),
\end{equation}
where $\mathcal{N}(0,1)$ is a standard normal law (i.e., with zero mean and unit variance), and $\stackrel{{\rm d}\mypp}{\to}$ denotes convergence in distribution.
\end{theorem}

\begin{proof}
Substituting \eqref{eq:tildeY}, the left-hand side of \eqref{eq:CLT1} is rewritten as
\begin{equation}\label{eq:Y*}
\varUpsilon(x)=\frac{1}{\sqrt{B\mypp\varphi_\nu(x)}}\sum_{i=1}^M \bigl(Z_i(A\myp x)-\bar{F}(A\myp x)\bigr).
\end{equation}
The characteristic function of
\eqref{eq:Y*} is given by
\begin{equation}\label{eq:psi}
\psi(t;x):=\EE\bigl(\rme^{\rmi\myp t \myp \varUpsilon(x)}\bigr)=\rme^{-\rmi\myp\tilde{t} \myp M \bar{F}(A\myp x)}\mypp \Bigl(1+\bar{F}(A\myp x)\myp\bigl(\rme^{\rmi\myp \tilde{t}}-1\bigr)\Bigr)^M,
\end{equation}
where
\begin{equation}\label{eq:tilde-t}
\tilde{t}= \frac{t}{\sqrt{B\mypp\varphi_\nu(x)}}\myp, \qquad  t\in\RR.
\end{equation}
Choosing the principal branch of the logarithm function $\CC\setminus\{0\}\ni z\mapsto \log z\in\CC$ (i.e., such that $\log 1=0$), we can rewrite \eqref{eq:psi}
as
\begin{equation}
\label{eq:log-psi}
\log\psi(t;x) =-\rmi\mypp \tilde{t}\mypp M \bar{F}(A\myp x) +M \log\myn(1+w),
\end{equation}
where
\begin{equation}\label{eq:w}
w:=\bar{F}(A\myp x)\myp\bigl(\rme^{\rmi\myp \tilde{t}}-1\bigr).
\end{equation}
Since $A\to\infty$ and $B\to\infty$ (by Assumptions  \ref{as:AB} and \ref{as:B}), we have $\tilde{t}\to0$ and $w\to0$, hence
$$
\log\myn(1+w)=w-\tfrac12\myp w^2+O(|w|^3).
$$
Therefore, Taylor expanding $\rme^{\rmi\myp \tilde{t}}=1+\rmi\myp \tilde{t}-\frac12\myp \tilde{t}^2+O(\tilde{t}^3)$ and substituting \eqref{eq:tilde-t} and \eqref{eq:w}, formula  \eqref{eq:log-psi} is elaborated as follows,
\begin{equation*}
\log\psi(t;x) =-\frac{M\bar{F}(A\myp x)\myp F(A\myp x)\,t^2}{2\myp B\mypp \varphi_\nu(x)}+O\!\left(\frac{M\bar{F}(A\myp x)}{B^{3/2}}\right)\to -\frac{\,t^2\!}{2},
\end{equation*}
using that $M\bar{F}(A\myp x)/B\sim \varphi_\nu(x)$ and $F(A\myp x)\to1$ for any $x>0$. Thus, $\psi(t;x)\to \rme^{-t^2\myn/2}$, which is the characteristic function of the normal distribution $\mathcal{N}(0,1)$, as claimed.
\end{proof}

Similarly, Theorem \ref{th:CLT1} can be extended to the finite-dimensional convergence.
\begin{theorem}\label{th:CLT2}
Under Assumptions \ref{as:theta}, \ref{as:AB} and \ref{as:B}, the random process\/ $(\varUpsilon(x),\myp x>0)$ defined in (\ref{eq:CLT1}) and (\ref{eq:Y*}), converges, in the sense of convergence of finite-dimensional distributions, to a Gaussian random process $(\varXi(x),\myp x>0)$ with zero mean and covariance function
\begin{equation*}
K(x,x'):=\Cov\bigl(\varXi(x),\varXi(x')\bigr)=\sqrt{\frac{\varphi_\nu(x')}{\varphi_\nu(x)}}\qquad (0<x\le x').
\end{equation*}
\end{theorem}
\begin{proof}
The proof proceeds along the same lines as in Theorem \ref{th:CLT1} via asymptotic analysis of the multivariate characteristic functions for any finite arrays $0<x_1\le \dots\le x_m$ ($m\in\NN$),
$$
\psi(t_1,\dots,t_m;x_1,\dots,x_m)=\EE\!\left(\rmi\sum_{k=1}^m t_k\myp\varUpsilon(x_k)\right)\qquad(t_1,\dots,t_m\in\RR).
$$
The limiting covariance function is easy to compute: using \eqref{eq:Y*} and \eqref{eq:Cov} we have, for $0<x\le x'$,
\begin{align*}
\Cov\bigl(\varUpsilon(x),\varUpsilon(x')\bigr)&=
\frac{M\bar{F}(A\myp x')\myp F(A\myp x)}{B\mypp\sqrt{\varphi_\nu(x)\mypp\varphi_\nu(x')}}\to \sqrt{\frac{\varphi_\nu(x')}{\varphi_\nu(x)}},
\end{align*}
as required.
\end{proof}

We conclude this section by observing that the limiting Gaussian process $\varXi(x)$ can be represented, in the distributional sense, through a standard Brownian motion $(B_t,\myp t\ge0)$ (i.e., with mean zero and covariance function $\Cov(B_t,B_{t'})=\min\myn\{t,t'\}$).
\begin{theorem}\label{th:CLT3}
The following distributional representation holds,
\begin{equation}\label{eq:B}
\varXi(x)\stackrel{{\rm d}}{=}\frac{B_{\varphi_\nu(x)}}{\sqrt{\varphi_\nu(x)}}\qquad (x>0).
\end{equation}
In particular, a rescaled process in inverted time\/ $t=1/x$, defined by
\begin{equation}\label{eq:BB}
\widetilde{\varXi}(t):=\sqrt{\varphi_\nu(1/t)}\:\varXi(1/t)\qquad (t\ge0),
\end{equation}
has independent increments.
\end{theorem}

\begin{proof}
It suffices to observe that the covariance function of the process on the right-hand side of \eqref{eq:B} is given by (for any $0<x\le x'$)
$$
\frac{1}{\sqrt{\varphi_\nu(x)\mypp \varphi_\nu(x')}}\,\Cov\bigl(B_{\varphi_\nu(x)},B_{\varphi_\nu(x')}\bigr)=\frac{1}{\sqrt{\varphi_\nu(x)\mypp \varphi_\nu(x')}}\,\varphi_\nu(x')=K(x,x'),
$$
using that $\varphi_\nu(x')\le \varphi_\nu(x)$. Finally, independence of increments for the process \eqref{eq:BB} is a straightforward consequence of the same property for the Brownian motion $B_t$.
\end{proof}

\section{Poisson approximation in the ``chaotic'' regime}\label{sec:6}

In this section, we consider the case wherein Assumption \ref{as:B} is not satisfied, so that the $y$-scaling coefficient $B$ is bounded (which is only possible for  $\nu\le0$, see formulas \eqref{eq:B1} to \eqref{eq:B4}). We call this case \emph{chaotic} because convergence of the random variable $\widetilde{Y}(x)$ to the limit shape $\varphi_\nu(x)$ does not hold here (cf.\:Theorem \ref{th:conv_Y}), despite convergence of the expected value $\EE\bigl(\widetilde{Y}(x)\bigr)=M\bar{F}(A\myp x)/B\to \varphi_\nu(x)$ (Theorem \ref{th:conv_exp}).
The root cause of this failure is that, although $\widetilde{Y}(x)$ is a normalized sum of independent Bernoulli variables $Z_i(A\myp x)=I_{\{X_i\ge Ax\}}$ (see \eqref{eq:tildeY}), the success probability $\PP(X_i\ge A\myp x)=\bar{F}(A\myp x)$ tends to zero, which is not offset by a fast enough growth of the number of terms $M$ (see Remark~\ref{rm:M}).

\subsection{One-dimensional distributions}\label{sec:6.1}
We start by studying one-dimensional distributions of $Y(x)$, that is, at a given point $x>0$. For orientation, consider a  stylized case where $B=1$, then $M\bar{F}(A\myp x)\to \varphi_\nu(x)$ and, according to the classic Poisson ``law of small numbers'' \cite{Whitaker}, the binomial distribution of the sum $\widetilde{Y}(x)=Z_1(A\myp x)+\dots+Z_M(A\myp x)$ is asymptotically close to a Poisson distribution with parameter $\lambda=\varphi_\nu(x)$.
That is to say, the sums $\widetilde{Y}(x)$ do not settle down to a deterministic constant (like in a law of large numbers) but, due to a persistent ``small'' randomness, admit a non-degenerate (Poisson) approximation without any normalization. This observation is generalized as follows.
\begin{theorem}\label{th:Poisson}
Suppose that Assumptions \ref{as:theta} and \ref{as:AB} are satisfied but
Assumption \ref{as:B} is not, so that $B=O(1)$. Then the distribution of the random variable\/ $Y(A\myp x)$ for $x>0$ is approximated by a Poisson distribution with parameter $M\bar{F}(A\myp x)\sim B\mypp\varphi_\nu(x)$ and with the corresponding error in total variation distance (or in Kolmogorov's uniform distance) bounded by $O(M^{-1})=o(1)$.
\end{theorem}
\begin{proof}
 This is an immediate consequence of a well-known approximation for the binomial distribution of the total number of successes in $n$  independent Bernoulli trials, with success probability $p$, by a Poisson distribution with \strut{}parameter $\lambda=np$, with the error bounded by $\sigma^2=n\myp p^2$ (see, e.g., \cite{Barbour,Novak}). In our case,
  \strut{}$\lambda=M\bar{F}(A\myp x)$ and $\sigma^2=M\bar{F}(A\myp x)^2\sim \bigl(B\mypp\varphi_\nu(x)\bigr)^2\mynn/M=O(1/M)=o(1)$.
\end{proof}

The Poisson approximation stated in Theorem \ref{th:Poisson} is illustrated in Fig.\,\ref{fig:GIGP-Poisson} using $100$ simulated samples (of size $M=35$ each) from the GIGP distribution \eqref{eq:GIGP} with parameters $\nu=-0.5$, $\alpha=2$, and $\theta=0.99$. The $y$-scaling coefficient computed from \eqref{eq:B3} is given by $B\doteq 1.974664$, confirming that this is a chaotic regime (i.e., where Assumption \ref{as:B} is not satisfied). The $x$-scaling coefficient \eqref{eq:A} specializes to $A\doteq 99.49916$. The left panel in Fig.\,\ref{fig:GIGP-Poisson} shows the sample Young diagrams superimposed on one another using transparent shading (in blue), so that darker places correspond to a more frequent occurrence. As anticipated, there is no convergence to a deterministic limit shape, but an emerging ``typical'' boundary of blue diagrams clearly indicates an expected curve establishing in the limit.

In the right panel of Fig.\,\ref{fig:GIGP-Poisson}, we choose a trial value $x_0=0.2$ and plot a histogram for the observed frequencies of the random values $Y(A\myp x_0)$, where $A\myp x_0\doteq 19.89983$. A visual inspection supports a reasonable match with Poisson distribution with the mean $M\bar{F}(A\myp x_0)\doteq4.342498$. This is confirmed by Pearson's $\chi^2$-test, with the bins labelled by the values of $j$ from $0$ to $9$ and the respective observed frequencies $o_j$. Since the expected frequencies $e_0\doteq 1.3004$ and $e_9\doteq 3.340067$ are less than $5$, we follow a common recommendation and combine the bins $j=0$ and $j=9$ with $j=1$ and $j=8$, respectively. The grouped $\chi^2$-statistic is calculated to yield $1.972246$ on $10-2-1=7$ degrees of freedom, with the $p$-value of $96.14\%$, so the goodness-of-fit test is comfortably passed.

\begin{figure}[!h!]
\renewcommand{\thesubfigure}{}\centering
\subfigure[\qquad (a)]
{\includegraphics[width=0.47\textwidth]{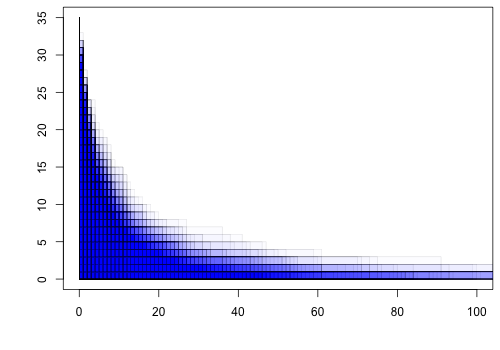}}
\put(-95,3){\mbox{\footnotesize$x$}}
\put(-203,96){\mbox{\footnotesize$y$}}
\hspace{1.9pc}\subfigure[\qquad (b)]
{\includegraphics[width=0.47\textwidth]{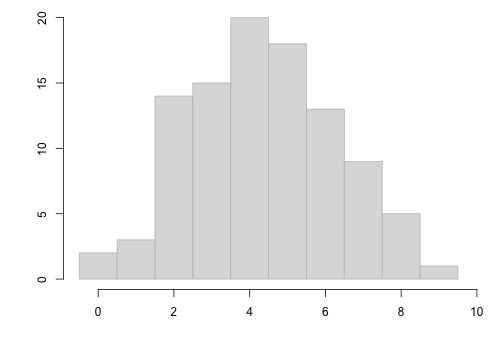}
\put(-95,3){\mbox{\footnotesize$j$}}
\put(-204,96)
{\mbox{\footnotesize$o_j$}}}
\caption{Illustration of Poisson statistics in the chaotic regime. The left panel shows superimposed Young diagrams of $100$ random samples of size $M=35$ each, generated from the GIGP model \eqref{eq:GIGP} with parameters $\nu=-0.5$, $\alpha=2$, and $\theta=0.99$.
The right panel shows the histogram of observed frequencies $o_j$ of the random values $Y(A\myp x_0)$, where $A\myp x_0\doteq 19.89983$. For orientation, the mean of the approximating Poisson distribution is given by $M\bar{F}(A\myp x_0)\doteq4.342498$.
}\label{fig:GIGP-Poisson}
\end{figure}

\subsection{Finite-dimensional distributions}\label{sec:6.2}
The result of Theorem \ref{th:Poisson} can be generalized to finite-dimensional  distributions. It is convenient to set $Y(\infty)=0$, using that, with probability $1$, $Y(x)\to0$ as $x\to\infty$.

\begin{theorem}\label{th:Poisson1}
Under the hypotheses of Theorem \ref{th:Poisson}, for any finite array\/ $0<x_1<\dots<x_k<x_{k+1}=\infty$, the increments\/ $\{Y(A\myp x_{i})-Y(A\myp x_{i+1}),\,i=1,\dots,k\}$ are asymptotically independent, with marginal distributions approximated by Poisson distributions with parameters $\lambda_{i}=M\bigl(\bar{F}(A\myp x_{i})-\bar{F}(A\myp x_{i+1})\bigr)$,
respectively, with the error (in total variation) bounded by
$O(M^{-1})=o(1)$.
\end{theorem}

\begin{remark}
The leftmost increment $Y(0)-Y(A\myp x_1)=M-Y(A\myp x_1)$ is excluded from the statement, because it is (linearly) expressible through the  other increments,
$$
Y(0)-Y(A\myp x_1)=M-\sum_{i=1}^k\bigl( Y(A\myp x_{i})-Y(A\myp x_{i+1})\bigr).
$$
\end{remark}

\begin{proof}[Proof of Theorem \ref{th:Poisson1}]
Note that the joint distribution of the increments is \emph{multinomial}, with parameter $M$ (the number of trials) and  probabilities $p_{i}=\bar{F}(A\myp x_{i})-\bar{F}(A\myp x_{i+1})$, corresponding to outcomes belonging to intervals $[A\myp x_{i},A\myp x_{i+1})$, respectively ($i=1,\dots,k$). The claim then follows by a general Poisson approximation theorem (see \cite{Novak}).  In particular, denoting $p:=p_1+\dots+p_k$, the error in total variation is known to be bounded by $O(M p^2)$. In our case, $p=\sum_{i=1}^k\bigl(\bar{F}(A\myp x_{i})-\bar{F}(A\myp x_{i+1})\bigr)=\bar{F}(A\myp x_1)$, and the estimate $O(M^{-1})$ readily follows like in Theorem \ref{th:Poisson}.
\end{proof}

\begin{remark}
To get a sense of why asymptotic independence of increments in Theorem \ref{th:Poisson1} is true, it is helpful to check out that the covariance between the neighboring increments vanishes in the limit. Indeed, for any $0<x<x'<x''\le\infty$ we have, with the aid of \eqref{eq:Cov},
\begin{align*}
\Cov\bigl(Y&(A\myp x)-Y(A\myp x'),\mypp Y(A\myp x')-Y(A\myp x'')\bigr)\\
&=
M\bar{F}(A\myp x')\mypp F(A\myp x)-M\bar{F}(A\myp x'')\mypp F(A\myp x)\\
&\quad -M\bar{F}(A\myp x')\mypp F(A\myp x')+M\bar{F}(A\myp x'')\mypp F(A\myp x')\\
&\to \varphi_\nu(x')-\varphi_\nu(x'')-\varphi_\nu(x')+\varphi_\nu(x'')=0,
\end{align*}
again using the convergence in Theorem \ref{th:conv_exp}.
\end{remark}

The finite-dimensional approximations of Theorem \ref{th:Poisson1} can be unified in terms of a suitable Poisson process considered in inverted time. Specifically, consider an inhomogeneous Poisson process $(\xi_t,\,t\ge 0)$, $\xi_0=0$, with integrated rate function $\varLambda(t)=M\bar{F}(A/t)\sim B\mypp\varphi_\nu(1/t)$, that is,
$\EE\myp(\xi_t)=\varLambda(t).$

\begin{theorem}\label{th:Poisson2} Under the hypotheses of Theorems \ref{th:Poisson} and \ref{th:Poisson1}, the distribution of the process $(Y(A\myp x),\,0<x\le \infty)$
is approximated by the distribution of the random process $(\xi_{1/x},\,0<x\le\infty)$.
\end{theorem}

Thus, unlike the ``regular'' case of Sections \ref{sec:4} and \ref{sec:5}, where the random Young diagrams enjoy convergence to the limit shape $\varphi_\nu(x)$, in the chaotic case that we consider in the present section, the role of the function
$\varphi_\nu(x)$ is that it determines the (integrated) rate of the Poisson approximation.

\section{Real Data Examples}\label{sec:7}
In this section, we look at how well the theoretical limit shape $\varphi_\nu(x)$ conforms to some real data sets studied earlier by Sichel \cite{Sichel1985}.

\subsection{Lotka's data set: author productivity}\label{sec:7.1}
We start with a classic data set considered by Lotka in his seminal paper \cite{Lotka}, comprising the counts of the number of papers (items) published by authors (sources) in \emph{Chemical Abstracts} during 1907--1916. This data set is usually considered as a baseline example of the power law statistics of counts \cite{CSN,Lotka}, but Sichel \cite[pp.\,316--317 and Table 2]{Sichel1985} argued that a GIGP model with predefined parameters $\nu=-0.5$, $\alpha=0$ and an estimated $\theta = 0.96876$ is a better fit to the data (with the $p$-value of $91.8\%$ in Pearson's $\chi^2$-test). The distinctive difference between the two models is of course the long-tail behavior, either power or power-geometric, respectively.

To examine the goodness-of-fit  graphically, similarly to Section \ref{sec:4.11} we first plot the empirical Young diagram $Y(x)$, compared with the fitted GIGP complementary distribution function $\bar{F}(x)$ and contrasted with the theoretical limit shape scaled back to the original coordinates, that is, $x\mapsto B\,\varphi(x/A)$. Here, $M=6\mypp 891$, and the scaling coefficients calculated from \eqref{eq:A} and \eqref{eq:B4} are given by $A\doteq 31.5076$ and $B\doteq 343.5839$. Although the value of $A$ is  not particularly large (because $\theta$ is not extremely close to $1$), a big value of $B$ confirms a reasonable predisposition of the data for a good limit shape approximation.
\begin{figure}[!h]
\renewcommand{\thesubfigure}{}
\centering
\hspace{-1pc}\subfigure[\raisebox{1.8pc}{(a) Lotka's data \cite{Lotka,Sichel1985}: GIGP parameters $\nu=-0.5$, $\alpha=0$ (predefined), $\theta=0.96876$ (estimated)}]%
{\includegraphics[width=0.47\textwidth]{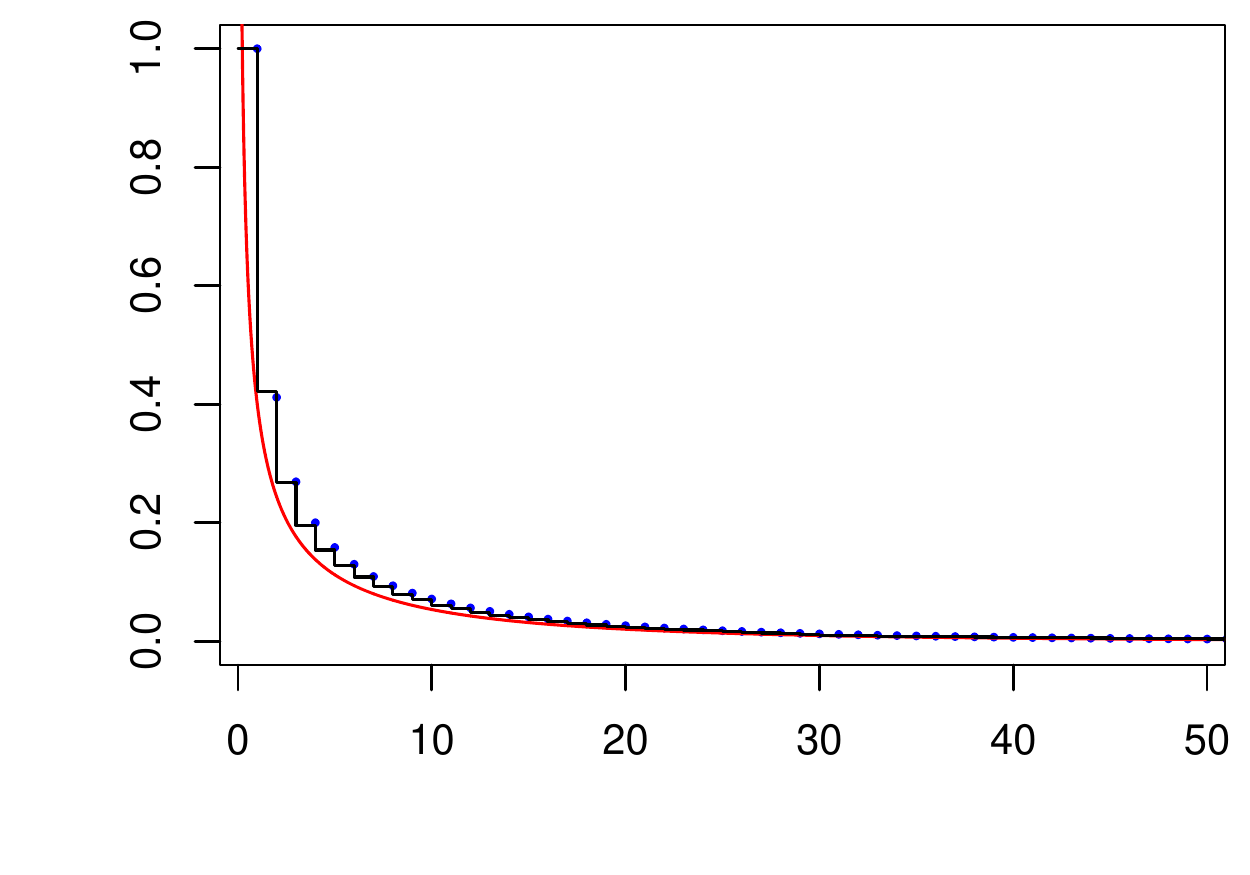}
\put(-90,6){\mbox{\footnotesize$x$}}
\put(-198,85){\mbox{\footnotesize$y$}}
\hspace{1pc}\includegraphics[width=0.47\textwidth]{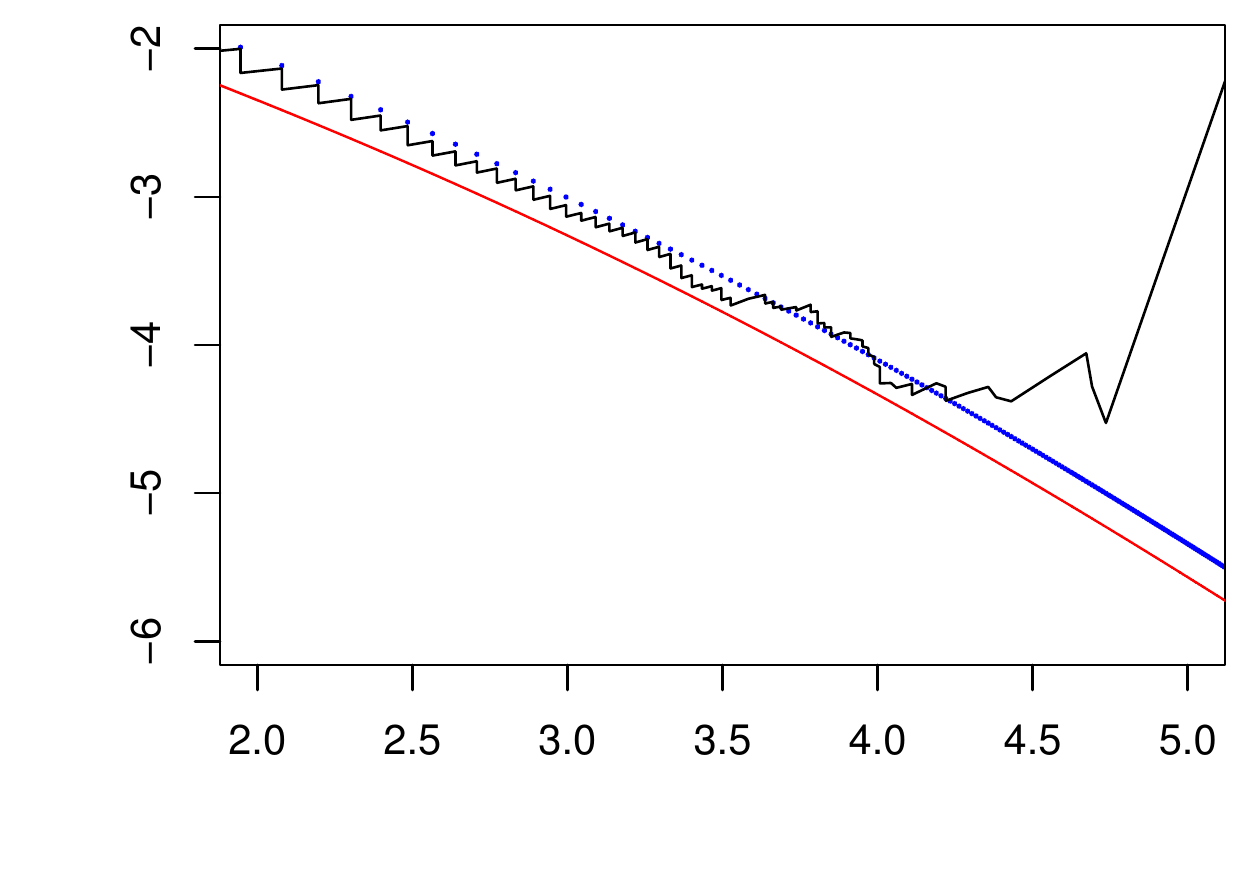}
\put(-90,6){\mbox{\footnotesize$u$}}
\put(-198,85){\mbox{\footnotesize$v$}}
}\\[-1.5pc]
\hspace{-1pc}\subfigure[\raisebox{1.8pc}{(b) Chen's data \cite{Chen,Sichel1985}: GIGP parameters $\nu=0$, $\alpha=0$ (predefined), $\theta=0.99369$ (estimated).}]%
{\includegraphics[width=0.47\textwidth]{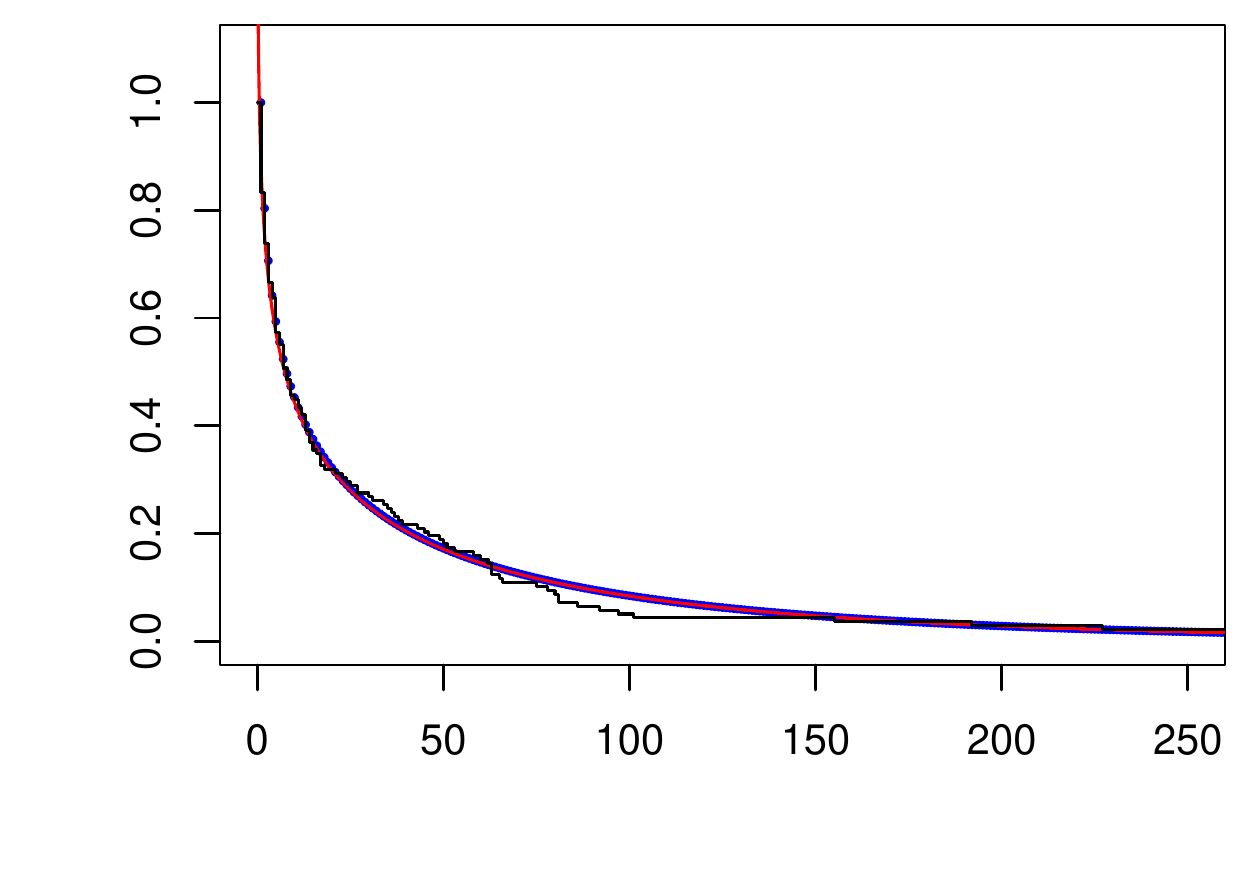}
\put(-90,6){\mbox{\footnotesize$x$}}
\put(-198,85){\mbox{\footnotesize$y$}}
\hspace{1pc}\includegraphics[width=0.47\textwidth]{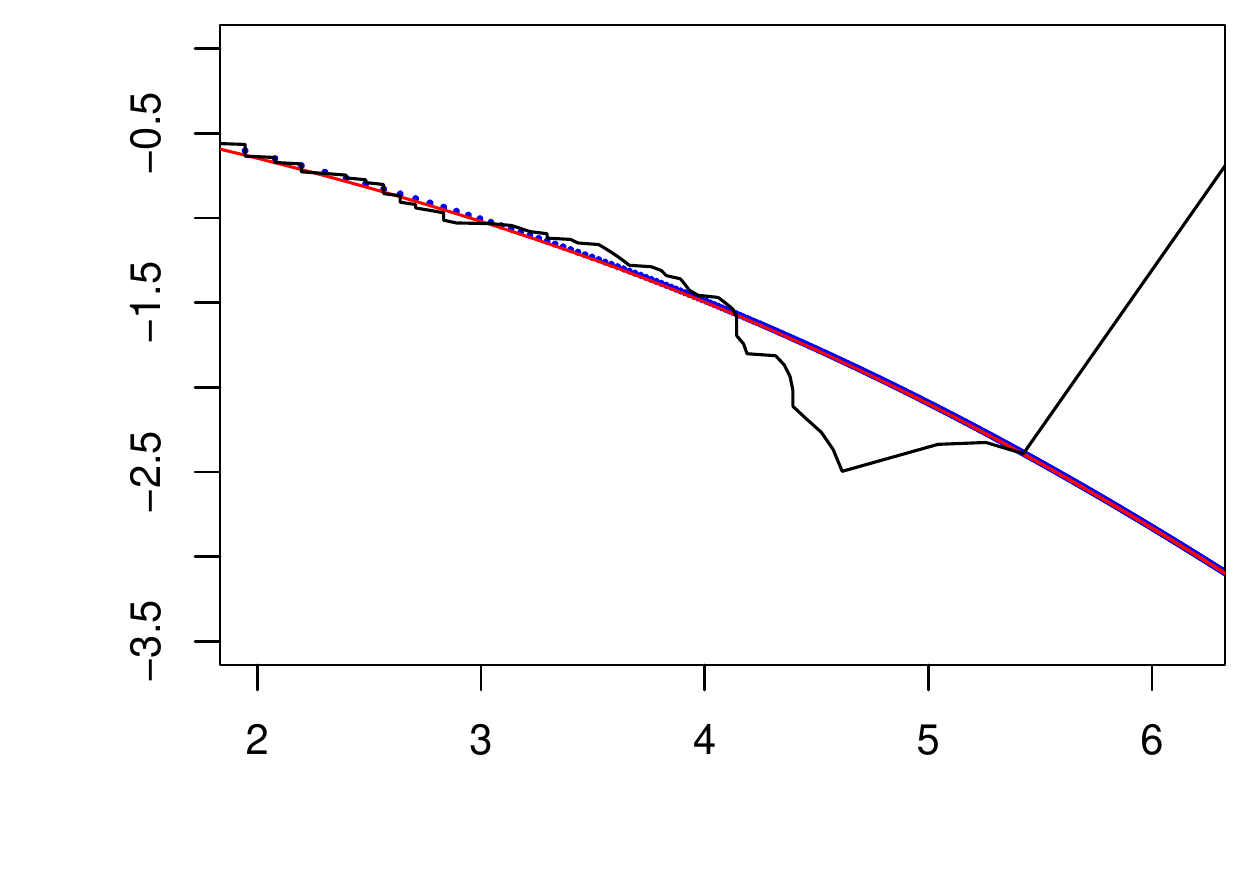}
\put(-90,6){\mbox{\footnotesize$u$}}
\put(-198,85){\mbox{\footnotesize$v$}}
}\vspace{-2pc}
\caption{GIGP model fit to real data sets. Black plots represent the data, blue dotted plots show the fitted GIGP complementary distribution functions, and smooth red lines depict the graphs of the scaled-back limit shape. The right panel shows the tail versions of these plots in transformed coordinates \eqref{eq:uv}.}
\label{Fig4:Lotkadata}
\end{figure}
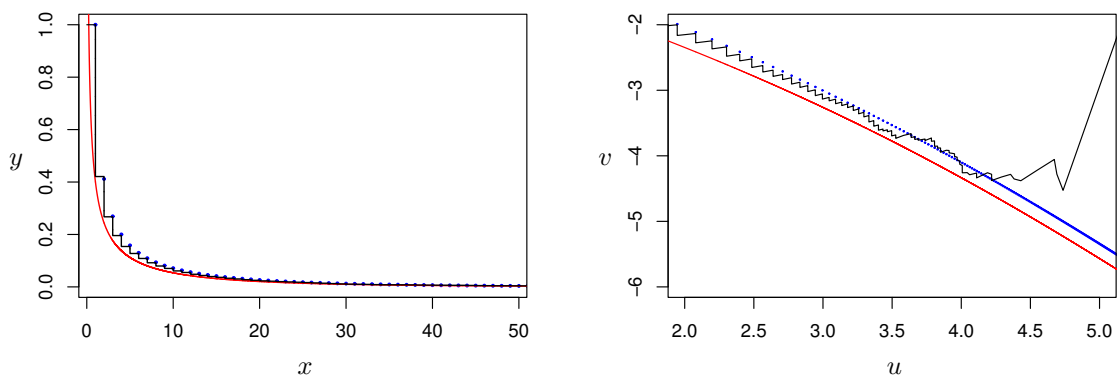
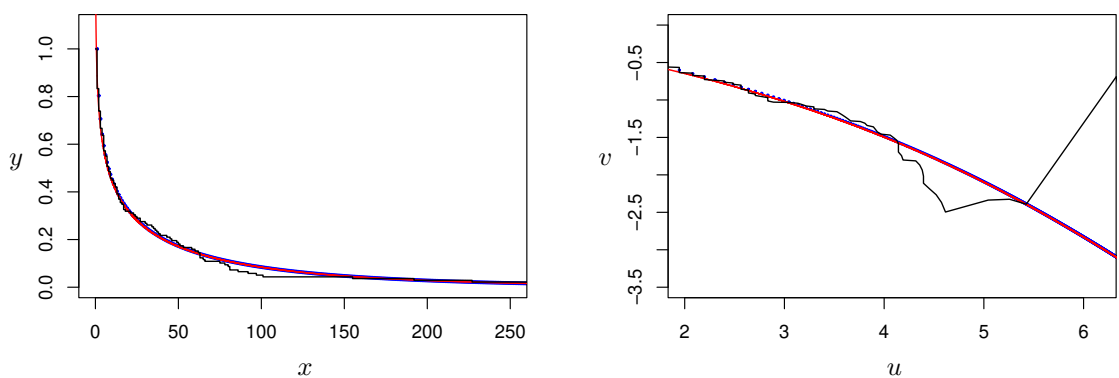

The left panel in Fig.\,\ref{Fig4:Lotkadata}\myp(a) demonstrates an excellent fit to the bulk of the data for both the GIGP model as well as the limit shape (scaled back to the original coordinates, $x\mapsto B\,\varphi(x/A)$). The visual inspection of the tails in the right panel of Fig.\,\ref{Fig4:Lotkadata}\myp(a) (in transformed coordinates \eqref{eq:uv}) confirms a good fit but only for moderately large observed values $j$, whereas the region $u\ge 4.5$, corresponding to values $j\ge \rme^{4.5}\approx 90$, reveals  increasing deviations from the GIGP prediction. This suggests that very large values in the tail of Lotka's data require a different fitting model, such as a stretched-exponential approximation \cite{Sornette}. Incidentally, upon a closer look at the upper extremes in the Lotka data set, there is just a handful larger than $90$, namely, $95$, $107$, $109$, $114$, and $346$. A surprising maximum $346$ looks like a genuine outlier, three times bigger than the runner-up! Interestingly, this record is attributed to Professor Emil Abderhalden, a prolific and controversial Swiss biochemist and physiologist who worked in the first half of the 20th century \cite{Wiki}. Being rather extraordinary, perhaps this individual record needs to be removed from statistical analysis.

\subsection{Chen's data set: journal use}\label{sec:7.2}
For our second real data example, we revisit the data set from Chen \cite{Chen}, considered by Sichel \cite[pp.\,318--319 and Table 4]{Sichel1985} for the sake of testing a GIGP model. The data comprised counts of use (items) of physics journals in the M.I.T.\ Science Library in 1971, recorded per each volume (sources) taken from the shelves for reading or photocopying. The total number of sources involved (i.e., the number of volumes ever requested) was $M=138$. Sichel fitted a GIGP model with predefined values $\nu=0$, $\alpha=0$ and an estimated  $\theta=0.99369$. He tested goodness-of-fit via Pearson's $\chi^2$-test, observing a reasonably high $p$-value of $31.2\%$, thus not signalling any significant mismatch.

To cross-examine the fit using our methods, similarly as in Section \ref{sec:7.1} we plot the empirical Young diagram's boundary $Y(x)$ along with the fitted GIGP function $\bar{F}(x)$ and scaled-back limit shape $B\,\varphi(x/A)$ (see Fig.\,\ref{Fig4:Lotkadata}\myp(b), left panel), where the scaling coefficients calculated from \eqref{eq:A} and \eqref{eq:B2} are given by $A\doteq 157.9781$ and $B\doteq 27.24247$. It is worth pointing out that, in contrast to Lotka's data set considered in Section \ref{sec:7.1}, here we have quite a large value for the $x$-scaling coefficient $A$ but a relatively small value of the $y$-scaling parameter $B$.

At first glance, the plots in Fig.\,\ref{Fig4:Lotkadata}\myp(b) (left panel) seem to conform to the GIGP model; however, one cannot help noticing a visible deviation from the theoretical prediction around the value $x=100$. This is confirmed by looking at the tail plots in transformed coordinates \eqref{eq:uv} (see Fig.\,\ref{Fig4:Lotkadata}\myp(b), right panel), where $x=100$ corresponds to $u=\log 100 \doteq 4.60517$. To test statistically whether the deviations are significant, we can use the asymptotic normality of $Y(x)$ due to Theorem \ref{th:CLT1}. Specifically, setting $x=100$ and stadardizing according to formula \eqref{eq:CLT1}, we calculate $\varUpsilon(100/A)\doteq -3.413073$, with an extremely small $p$-value of $0.032\%$. Thus, the deviation is highly significant, which implies that the GIGP model is not an accurate fit, at least for moderately large values starting from about $x=70$.

\section{Conclusion}\label{sec:8}

In this paper, we have investigated the asymptotic convergence of the item production profile (visualized via Young diagrams and properly scaled) to the limit shape for the class of GIGP distributions, introduced by Sichel \cite{Sichel1971} and successfully applied to a plurality of count data sets including frequent use cases in informetrics \cite{Sichel1985}. The limit is taken when the number of production sources is large, $M\gg1$, but also subject to a natural assumption of asymptotically large number of items per source, leading to the parametric regime $\theta\approx1$, where $\theta\in(0,1)$ is one of the GIGP scale parameters.

The family of the resulting limit shapes has been identified as the incomplete gamma \strut{}function $\varphi_\nu(x)=\int_x^\infty\! s^{\nu-1}\mypp\rme^{-s}\,\rmd{s}$, where $\nu\ge -1$ is the GIGP shape parameter (represented as the order of the \strut{}Bessel \strut{}function $K_\nu(z)$ involved in the definition of the GIGP probability distribution). The suitable $x$-scaling coefficient is universal, $A=-1/\log\theta\to\infty$, but the $y$-scale $B$ (proportional to the number of sources $M$) depends on the sign of $\nu$. In particular, it follows that $B\to\infty$ if $\nu>0$ (which guarantees convergence to the limit shape), but this may fail for $\nu\le0$, for instance, due to $M$ not being large enough. In the latter case (termed ``chaotic'') we show that the production profile is approximated by a suitable Poisson process with rate defined in terms of the limit shape $\varphi_\nu(x)$. In the regular case (i.e, where the convergence to limit shape holds true),
we also show the asymptotic normality of fluctuations.

Our theoretical results are illustrated using computer simulations (with $\nu=0.5$ and $\nu=-0.5$), showing an excellent match of the limit shape to empirical data in regular cases but also confirming the Poisson statistics in a chaotic regime. We also propose a simple transformation of the data leading to linearized tails of the distribution, which may aid the visual diagnostics of tenability of the GIGP model.

When applied to  real-life data sets, our methods provide a novel approach to extracting useful information about the GIGP model fit. One such examples is the classic Lotka data on author productivity \cite{Lotka}, where the fitted GIGP model enjoys an excellent match, which is conveniently grasped by our graphical plots. However, the plots of properly enhanced upper tails  reveal a certain departure from the GIGP model, which is not captured by the usual aggregated goodness-of-fit tools such as Pearson $\chi^2$-test or Kolmogorov--Smirnov test based on uniform distance between probability distributions \cite{Sichel1985}. The identified extremes can be tracked back and interpreted in the original data, in particular flagging up a very special ``outlier'' as mentioned in Section \ref{sec:7.1}.

Furthermore, in our second example using Chen's data on journal usage \cite{Chen}, our approach has revealed a certain location in the range of counts with a noticeable deviation of the data frequency plot from the theoretical GIGP prediction. By virtue of our result on asymptotic normality of fluctuations, we have shown that these deviations are highly significant. This is to be compared with a reasonably confident ``pass'' of goodness-of-fit based on traditional tools \cite{Sichel1985}, again demonstrating superiority of our methods. We expect that the approach developed in this paper is likely to be useful also in a wider context of various count data.

\appendix

\section{Asymptotic formulas for the Bessel function}\label{sec:A}
The following is a list of useful properties of the Bessel function $K_\nu(z)$, including some asymptotic formulas under various regimes for the argument $z$ and the order $\nu$. For ease of use, we collect these facts here, with reference to the NIST handbook \cite{NIST}.

\begin{lemma}[\cite{NIST}, 10.27.3]
For any $\nu$ and $z$,
\begin{equation}\label{eq:K1}
K_{-\nu}(z)=K_\nu(z).
\end{equation}
\end{lemma}

\begin{lemma}
Let $\nu$ be fixed and $z\to0+$.
\begin{itemize}
\item[\rm (a)]
\textup{(\cite{NIST}, 10.30.2)} If $\nu>0$ then
\begin{equation}\label{eq:K2}
K_\nu(z)\sim \tfrac12\mypp \Gamma(\nu) \left(\tfrac12\myp z\right)^{-\nu}\!\mynn.
\end{equation}
\item[\rm (b)]
\textup{(\cite{NIST}, 10.31.1 with the aid of 10.25.2)}
If $\nu=1$ then
\begin{equation}\label{eq:K31}
K_1(z)=z^{-1} + \tfrac12\myp z
\log z +O(z).
\end{equation}

\item[\rm (c)]
\textup{(\cite{NIST}, 10.31.2 with the aid of 10.25.2)} If $\nu=0$ then
\begin{equation}\label{eq:K3}
K_0(z)=-\log\mynn \bigl(\tfrac12\myp z\bigr)-\gamma+O(z^2\log z).
\end{equation}
where $\gamma= 0.5772\dots$
is Euler's constant (\cite{NIST}, 5.2.3).
In particular,
\begin{equation}\label{eq:K30}
K_0(z)\sim -\log z.
\end{equation}

\item[\rm (d)] \textup{(\cite{NIST}, 10.27.4 and 10.25.2 with the aid of 5.5.3)}
For $-1<\nu<0$,
\begin{equation}\label{eq:K4}
K_\nu(z)=\tfrac12\mypp \Gamma(-\nu) \left(\tfrac12\myp z\right)^{\nu}+\frac{\Gamma(\nu+1)}{2\myp\nu}\left(\tfrac12\myp z\right)^{-\nu}+O(z^{\nu+2}).
\end{equation}
\end{itemize}
\end{lemma}

\begin{lemma}[\cite{NIST}, 10.41.2]
If $z\ne0$ is fixed and $\nu\to+\infty$, then
\begin{equation}\label{eq:K5}
K_\nu(z)\sim\sqrt{\frac{\pi}{2\myp\nu}} \left(\frac{\rme\myp z}{2\myp\nu}\right)^{-\nu}\!\mynn.
\end{equation}
\end{lemma}


\begin{thebibliography}{99}

\bibitem{ABT}
Arratia, R., Barbour, A.D., Tavar\'e, S. \emph{Logarithmic Combinatorial Structures: A Probabilistic Approach}. EMS Monographs in Mathematics. European Mathematical Society: Z\"urich,  2003; \href{https://doi.org/10.4171/000}{doi:\allowbreak 10.\allowbreak 4171/\allowbreak 000}.

\bibitem{Auluck}
Auluck, F.C., Kothari, D.S. Statistical mechanics and the
partitions of numbers. \textit{Proc.\ Cambridge Philos.\ Soc.}\ {\bf 42} (1946),
272--277; \href{https://doi.org/10.1017/S0305004100023033}{doi:\allowbreak
10.\allowbreak 1017/\allowbreak S030500410002303}.

\bibitem{Barbour}
Barbour, A.D., Holst, L., Janson, S. \textit{Poisson Approximation}. Oxford Studies in Probability, {\bf 2}.  Oxford Science Publications. Clarendon Press: Oxford, 1992.

\bibitem{Barndorff}
Barndorff-Nielsen, O.E., Bl\ae sild, P.,  Seshadri, V.
Multivariate distributions with generalized inverse Gaussian marginals, and associated Poisson mixtures. \textit{Canad.\ J.\ Statist.}
{\bf 20} (1992),
109--120; \href{https://doi.org/10.2307/3315462}{doi:\allowbreak  10.\allowbreak 2307/\allowbreak 3315462}.

\bibitem{bogachev2015unified}
Bogachev, L.V. Unified derivation of the limit shape for multiplicative ensembles of random integer partitions with equiweighted parts. \emph{Random Struct.\ Algorithms} {\bf 47} (2015),
227--266; \href{https://doi.org/10.1002/rsa.20540}{doi:\allowbreak 10.\allowbreak 1002/\allowbreak rsa.\allowbreak 20540}.

\bibitem{BGY}
Bogachev, L.V., Gnedin, A.V., Yakubovich, Yu.V.
On the variance of the number
of occupied boxes. \textit{Adv.\ in Appl.\ Math.} {\bf 40} (2008), 401--432; \href{https://doi.org/10.1016/j.aam.2007.05.002}{doi:\allowbreak 10.\allowbreak 1016/\allowbreak j.\allowbreak aam.\allowbreak 2007.\allowbreak 05.\allowbreak 002}.

\bibitem{Borisov}
Borisov, I., Jetpisbaev, M. Poissonization principle for a class of additive statistics. \emph{Mathematics} {\bf 10} (2022), 4084; \href{https://doi.org/10.3390/math10214084}{doi:\allowbreak 10.\allowbreak 3390/\allowbreak math\allowbreak 10214084}

\bibitem{BLM}
Boucheron, S., Lugosi, G., Massart, P. \textit{Concentration Inequalities: A Nonasymptotic Theory of Independence}. Oxford University Press:  Oxford, 2013; \href{https://doi.org/10.1093/acprof:oso/9780199535255.001.0001}{doi:\allowbreak 10.\allowbreak 1093/\allowbreak acprof:\allowbreak oso/\allowbreak 9780199535255.\allowbreak 001.\allowbreak 0001}.




\bibitem{Chen}
Chen, C.-C. The use patterns of physics journals in a large academic research library. \emph{J.~Amer.\ Soc.\ Inf. Sci.} {\bf 23} (1972), 254--270; \href{https://doi.org/10.1002/asi.4630230405}{doi:\allowbreak 10.\allowbreak 1002/\allowbreak asi.\allowbreak 4630230405}.

\bibitem{CSN}
Clauset, A., Shalizi, C.R., Newman, M.E.J.  Power-law distributions in empirical data. \textit{SIAM
Rev.} {\bf 51} (2009), 661--703; \href{https://doi.org/10.1137/070710111}{doi:\allowbreak 10.\allowbreak 1137/\allowbreak 070710111}.

\bibitem{Coile}
Coile, R.C. Lotka's frequency distribution of scientific productivity. \textit{J.~Amer.\
Soc.\ Inf.\ Sci.} {\bf 28} (1977), 366--370;  \href{https://doi.org/10.1002/asi.4630280610}{doi:\allowbreak 10.\allowbreak 1002/\allowbreak asi.\allowbreak 4630280610}.

\bibitem{Cramer}
Cram\'er, H.  \textit{Mathematical Methods of Statistics}.  Princeton
Mathematical Series, {\bf 9}. Princeton University Press: Princeton,
NJ, 1946; \href{https://www.jstor.org/stable/j.ctt1bpm9r4}{https://\allowbreak www.\allowbreak jstor.\allowbreak org/\allowbreak stable/\allowbreak j.\allowbreak ctt\allowbreak 1bpm9r4}.

\bibitem{Duchon}
Duchon, P., Flajolet, P., Louchard, G.,  Schaeffer, G. Boltzmann samplers for the random generation of combinatorial structures.
\emph{Combin.\ Probab.\ Comput.} {\bf 13} (2004), 577--625; \href{https://doi.org/10.1017/S0963548304006315}{doi:\allowbreak 10.1017/\allowbreak S0963548304006315}.

\bibitem{Egghe}
Egghe, L. \textit{Power Laws in Information Production Processes: Lotkaian
Informetrics}. Library and Information Science, {\bf 5}. Elsevier Academic Press: Amsterdam, 2005; Emerald: Bingley, 2005; \href{https://doi.org/10.1108/S1876-0562(2005)0000005004}{doi:\allowbreak 10.\allowbreak 1108/\allowbreak S1876-\allowbreak 0562\allowbreak (2005)\allowbreak 0000005004}.

\bibitem{ER}
Egghe, L., Rousseau, R. Generalized success-breeds-success
principle leading to time-dependent informetric distributions. \emph{J.~Amer.\ Soc.\ Inf.\ Sci.} {\bf 46} (1995), 426--445; \href{https://doi.org/10.1002/(SICI)1097-4571(199507)46:6<426::AID-ASI3>3.0.CO;2-B}{doi:\allowbreak 10.\allowbreak 1002/\allowbreak (SICI)\allowbreak 1097-\allowbreak 4571\allowbreak (199507)\allowbreak 46:\allowbreak 6<426::\allowbreak AID-\allowbreak ASI\allowbreak 3\allowbreak \textgreater3.\allowbreak 0.\allowbreak CO;\allowbreak 2-B}.

\bibitem{Fristedt}
Fristedt, B. The structure of random partitions of large integers. \emph{Trans.\ Amer.\ Math.\ Soc.}\ {\bf 337} (1993), 703--735; \href{https://doi.org/10.1090/S0002-9947-1993-1094553-1}{doi:\allowbreak
10.\allowbreak 1090/\allowbreak S0002-9947-1993-1094553-1}.

\bibitem{Gnedin}
Gnedin, A., Hansen, B., Pitman, J. Notes on the occupancy problem with infinitely many boxes: general asymptotics and power laws. \textit{Probab.\ Surveys} {\bf 4} (2007), 146--171; \href{https://doi.org/10.1214/07-PS092}{doi:\allowbreak 10.\allowbreak 1214/\allowbreak 07-PS092}.

\bibitem{Gupta}
Gupta, R.C., Ong, S.H. Analysis of long-tailed count data by Poisson mixtures. \emph{Comm.\ Statist.\ Theory Methods} {\bf  34} (2005),
557--573; \href{https://doi.org/10.1081/STA-200052144}{doi:\allowbreak 10.\allowbreak 1081/\allowbreak STA-200052144}.

\bibitem{Hirsch}
Hirsch, J.E. An index to quantify an individual's scientific research output. \emph{Proc.\ Natl.\ Acad.\ Sci.\ USA} {\bf 102} (2005),
16569--16572;  \href{https://doi.org/10.1073/pnas.0507655102}{doi:\allowbreak 10.\allowbreak 1073/\allowbreak pnas.\allowbreak 0507655102}.

\bibitem{Huber}
Huber, J.C. A new model that generates Lotka's law. \emph{J.~Amer.\ Soc.\ Inf.\
Sci.} {\bf 53} (2002), 209--219; \href{https://doi.org/10.1002/asi.10025}{doi:\allowbreak 10.\allowbreak 1002/\allowbreak asi.\allowbreak 10025}.

\bibitem{Johnson1}
Johnson, N.L., Kemp, A.W., Kotz, S. \emph{Univariate Discrete Distributions}, 3rd ed.. Wiley Series in Probability and Mathematical Statistics.
Wiley: New York, 2005; \href{https://doi.org/10.1002/0471715816}{doi:\allowbreak 10.\allowbreak 1002/\allowbreak 0471715816}.

\bibitem{Johnson}
Johnson, N.L., Kotz, S.,  Balakrishnan, N. \emph{Continuous Univariate Distributions, Vol.\,1}, 2nd ed. Wiley Series in Probability and Mathematical Statistics.
Wiley: New York, 1994.

\bibitem{Karlin}
Karlin, S. Central limit theorems for certain infinite urn schemes. \emph{J.~Math.\ Mech.} {\bf 17} (1967),
373--401; \href{https://www.jstor.org/stable/24902077}{https://\allowbreak www.\allowbreak jstor.\allowbreak org/\allowbreak stable/\allowbreak 24902077}.

\bibitem{Sornette}
Laherr\`ere, J., Sornette, D. Stretched exponential distributions in nature and economy: ``fat tails'' with characteristic scales. \textit{Eur.\ Phys.\ J.\ B} {\bf 2} (1998), 525--539; \href{https://doi.org/10.1007/s100510050276}{doi:\allowbreak 10.\allowbreak 1007/\allowbreak s100510050276}.

\bibitem{Lotka}
Lotka, A.J. The frequency distribution of scientific productivity. \textit{J.~Washington Acad.\ Sci.} {\bf 16} (1926), 317--323;  \href{http://www.jstor.org/stable/24529203}{http://\allowbreak www.\allowbreak jstor.\allowbreak org/\allowbreak stable/\allowbreak 24529203}.

\bibitem{Nicholls}
Nicholls, P.T. Estimation of Zipf parameters. \emph{J.~Amer.\ Soc.\ Inf.\ Sci.} {\bf 38} (1987),
443--445; \href{https://doi.org/10.1002/(SICI)1097-4571(198711)38:6<443::AID-ASI4>3.0.CO;2-E}{doi:\allowbreak 10.\allowbreak 1002/\allowbreak (SICI)\allowbreak 1097-4571\allowbreak (198711)\allowbreak 38:\allowbreak 6\textless 443\allowbreak ::\allowbreak AID-ASI4\textgreater 3.\allowbreak 0.\allowbreak CO;\allowbreak 2-E}; Errata: Kraft, D.H.  \textit{Ibidem} {\bf 39} (1988),
287;
\href{https://doi.org/10.1002/(SICI)1097-4571(198807)39:4<287::AID-ASI10>3.0.CO;2-T}{doi:\allowbreak 10.\allowbreak 1002/\allowbreak (SICI)\allowbreak 1097-\allowbreak 4571\allowbreak (198807)\allowbreak 39:\allowbreak 4\textless 287\allowbreak ::\allowbreak AID-ASI10\allowbreak \textgreater \allowbreak 3.\allowbreak 0.CO;2-T}.

\bibitem{NIST}
\emph{NIST Handbook of Mathematical Functions}, Olver, F.W.J.\ et al.,
Eds. National Institute of Standards and Technology, U.S.\ Department of Commerce: Washington, DC; Cambridge University Press: Cambridge, 2010; \href{https://www.cambridge.org/catalogue/catalogue.asp?isbn=9780521192255}{https://\allowbreak www.\allowbreak cambridge.\allowbreak org/\allowbreak catalogue/\allowbreak catalogue.\allowbreak asp?\allowbreak isbn\allowbreak =\allowbreak 9780521192255};
Online version:\ \emph{NIST Digital Library of Mathematical Functions}, Version 1.1.8, Release date 2022-12-15; \href{https://dlmf.nist.gov}{https://\allowbreak dlmf.\allowbreak nist.\allowbreak gov} (accessed January 8, 2023).


\bibitem{Novak}
Novak, S.Y. Poisson approximation. \emph{Probab.\ Surveys} {\bf 16} (2019), 228--276; \href{https://doi.org/10.1214/18-PS318}{doi:\allowbreak 10.\allowbreak 1214/\allowbreak 18-PS318}.

\bibitem{NBV}
Nuermaimaiti, R., Bogachev, L.V., Voss, J. A generalized power
law model of citations. In \emph{18th International Conference on
Scientometrics and Informetrics (ISSI 2021)}, Proceedings (Leuven,
2021), Gl\"anzel, W.\ et al., Eds.  International Society for Scientometrics and
Informetrics (I.S.S.I.), 2021, pp.\;843--848;
\href{https://kuleuven.app.box.com/s/kdhn54ndlmwtil3s4aaxmotl9fv9s329}{https://\allowbreak issi2021.\allowbreak org/\allowbreak proceedings/}.

\bibitem{Pittel}
Pittel, B. On a likely shape of the random Ferrers diagram.
\textit{Adv.\ in Appl.\ Math.} {\bf 18} (1997),
432--488; \href{https://doi.org/10.1006/aama.1996.0523}{doi:\allowbreak 10.\allowbreak 1006/\allowbreak aama.\allowbreak 1996.\allowbreak 0523}.

\bibitem{Price1965}
Price, D.J.S. Networks of scientific papers. \emph{Science} {\bf 149} (1965),
510--515; \href{https://doi.org/10.1126/science.149.3683.510}{doi:\allowbreak 10.\allowbreak 1126/\allowbreak science.\allowbreak 149.\allowbreak 3683.\allowbreak 510}.

\bibitem{Price}
Price, D.J.S. A general theory of bibliometric and other cumulative
advantage processes. \textit{J.~Amer.\ Soc.\ Inf.\
Sci.} {\bf 27} (1976), 292--306; \href{https://doi.org/10.1002/asi.4630270505}{doi:\allowbreak 10.\allowbreak 1002/\allowbreak asi.\allowbreak 4630270505}.

\bibitem{Redner}
Redner, S. How popular is your paper? An empirical study of the citation distribution. \emph{Eur.\ Phys.\ J.~B}
{\bf 4} (1998),
131--134;
\href{https://doi.org/10.1007/s100510050359}{doi:\allowbreak 10.\allowbreak 1007/\allowbreak s100510050359}.


\bibitem{Resnick}
Resnick, S.I. \textit{Extreme Values, Regular Variation and Point Processes}.
Applied Probability. A Series of the Applied Probability Trust, {\bf 4}. Springer: New York, 1987; \href{https://doi.org/10.1007/978-0-387-75953-1}{doi:\allowbreak 10.\allowbreak1007/\allowbreak 978-\allowbreak 0-\allowbreak 387-\allowbreak 75953-1}.

\bibitem{Shiryaev}
Shiryaev, A.N. \emph{Probability}, 2nd ed. Graduate Texts in
Mathematics, {\bf 95}.  Springer: New York, 1996;
\href{https://doi.org/10.1007/978-1-4757-2539-1}{doi:\allowbreak 10.\allowbreak 1007/\allowbreak 978-1-4757-2539-1}.

\bibitem{Sichel1971}
Sichel, H.S. On a family of discrete distributions particularly suited to
represent long-tailed frequency data. In \emph{Proceedings of the Third Symposium on Mathematical Statistics}, Laubscher, N.F., Ed. Council for Scientific and Industrial Research (CSIR): Pretoria, 1971, pp.\;51--97.

\bibitem{Sichel1973a}
Sichel, H.S.  Statistical valuation of diamondiferous deposits. \emph{J.~South Afr.\ Inst.\ Min.} {\bf 73} (1973),
235--243;
\href{https://hdl.handle.net/10520/AJA0038223X_302}{https://\allowbreak hdl.\allowbreak handle.\allowbreak net/\allowbreak 10520/\allowbreak AJA\allowbreak  00\allowbreak 38\allowbreak 223X\_302}.


\bibitem{Sichel1974}
Sichel, H.S. On a distribution representing sentence-length in written prose. \textit{J.~Roy.\ Statist.\ Soc.\ Ser.~A} {\bf 137} (1974), 25--34; \href{https://doi.org/10.2307/2345142}{doi:\allowbreak 10.\allowbreak 2307/\allowbreak 2345142}.

\bibitem{Sichel1975}
Sichel, H.S. On a distribution law for word frequencies. \emph{J.~Amer.\ Statist.\ Assoc.} {\bf 70} (1975),
542--547;
\href{https://doi.org/10.1080/01621459.1975.10482469}{https://\allowbreak doi:\allowbreak 10.\allowbreak 1080/\allowbreak 01621459.\allowbreak 1975.\allowbreak 10482469}.

\bibitem{Sichel1982}
Sichel, H.S. Repeat-buying and the generalized inverse Gaussian-Poisson distribution. \emph{J.~Roy.\ Statist.\ Soc.\ Ser.~C} {\bf 31} (1982),
193--204;
\href{https://doi.org/10.2307/2347993}{doi:\allowbreak 10.\allowbreak 2307/\allowbreak 2347993}.

\bibitem{Sichel1985}
Sichel, H.S. A bibliometric distribution which really works. \emph{J.~Amer.\ Soc.\ Inf.\ Sci.} {\bf 36} (1985),
314--321; \href{https://doi.org/10.1002/asi.4630360506}{doi:\allowbreak 10.\allowbreak 1002/\allowbreak asi.\allowbreak 4630360506}.

\bibitem{Sichel1997}
Sichel, H.S. Modelling species-abundance frequencies and species-individual functions with the generalized inverse Gaussian-Poisson distribution. \emph{South African Statist.\,J.} {\bf 31} (1997),
13--37; \href{https://hdl.handle.net/10520/AJA0038271X_555}{https://\allowbreak hdl.\allowbreak handle.\allowbreak net/\allowbreak 10520/\allowbreak AJA0038271X\_555}.

\bibitem{V0}
Vershik, A.M. Asymptotic combinatorics and algebraic
analysis. In \emph{Proceedings of the International Congress of
Mathematicians (Z\"urich, 1994)},
Vol.~2, Chatterji, S.D., Ed.
Birkh\"auser: Basel, 1995, pp.\mypp 1384--1394;
\href{https://doi.org/10.1007/978-3-0348-9078-6_133}{doi:\allowbreak 10.1007/\allowbreak 978-3-0348-9078-6\_133}.

\bibitem{vershik1996statistical}
Vershik, A.M. Statistical mechanics of combinatorial partitions, and
their limit shapes.
\emph{Funct.\ Anal.\ Appl.} {\bf 30} (1996),
90--105; \href{https://doi.org/10.1007/BF02509449}{doi:\allowbreak 10.\allowbreak 1007/\allowbreak BF02509449}.

\bibitem{Vershik2}
Vershik, A.M. Limit distribution of the energy of a quantum ideal
gas from the viewpoint of the theory of partitions of natural
numbers.
\textit{Russian Math.\ Surveys}
{\bf 52} (1997), 379--386; \href{https://doi.org/10.1070/RM1997v052n02ABEH001782}{doi:\allowbreak 10.\allowbreak 1070/\allowbreak RM1997\allowbreak v052n02ABEH001782}.

\bibitem{Whitaker}
Whitaker, L. On the Poisson  law of small numbers. \emph{Biometrika} {\bf 10} (1914), 36--71; \href{https://doi.org/10.1093/biomet/10.1.36}{doi:\allowbreak 10.\allowbreak
1093/\allowbreak biomet/\allowbreak 10.1.36}.

\bibitem{Wiki}
Wikipedia, The Free Encyclopedia. ``Emil Abderhalden''. Available online: \href{https://en.wikipedia.org/wiki/Emil_Abderhalden}{https://\allowbreak en.\allowbreak wikipedia.\allowbreak org/\allowbreak wiki/\allowbreak Emil\_\allowbreak Abderhalden} (accessed March 10, 2023).

\bibitem{Yeh} Yeh, J.\ \textit{Martingales and Stochastic Analysis}. Series on Multivariate Analysis, {\bf 1}. World
Scientific: Singapore, 1995; \href{https://doi.org/10.1142/2948}{doi:\allowbreak 10.\allowbreak 1142/\allowbreak 2948}.

\bibitem{Yong}
Yong, A. Critique  of  Hirsch's  citation  index:\  A  combinatorial Fermi  problem. \emph{Notices Amer.\ Math.\ Soc.} {\bf 61} (2014), 1040--1050; \href{https://doi.org/10.1090/noti1164}{doi:\allowbreak 10.\allowbreak 1090/\allowbreak noti1164}.

\end{thebibliography}

\end{document}